\documentclass[leqno,final]{siamltex}

\usepackage{graphicx}
\usepackage{epstopdf, color}
\usepackage{amsmath,amstext,amssymb,hyperref}

\numberwithin{equation}{section}

\newcommand{\norm}[1]{\left\Vert#1\right\Vert}

\newcommand{\norme}[1]{\left\Vert{\hskip -2.7pt}\left\vert #1 \right\vert{\hskip -2.7pt}\right\Vert}

\newcommand{\abs}[1]{\left\vert#1\right\vert}
\newcommand{\pd}[1]{\left\langle #1\right\rangle}

\newcommand{\set}[1]{\left\{#1\right\}}

\newcommand{\jm}[1]{\left[#1\right]}

\newcommand{\R}{\mathbb{R}}

\newcommand{\nn}{\nonumber}
\newcommand{\ls}{\lesssim}

\newcommand{\al}{\alpha}
\newcommand{\be}{\beta}

\newcommand{\De}{\Delta}
\newcommand{\ep}{\varepsilon}
\newcommand{\ga}{\gamma}
\newcommand{\Ga}{\Gamma}

\newcommand{\na}{\nabla}
\newcommand{\om}{\omega}
\newcommand{\Om}{\Omega}
\newcommand{\pa}{\partial}

\newcommand{\ta}{\theta}

\newcommand{\dx}{\,\mathrm{d} x}

\newcommand{\M}{\mathcal{M}}
\newcommand{\N}{\mathcal{N}}

\renewcommand{\i}{{\rm\mathbf i}}
\newcommand{\bt}{{\rm\mathbf t}}
\newcommand{\bn}{{\rm\mathbf n}}
\newcommand{\bm}{{\rm\mathbf m}}

\newcommand{\T}{\mathcal{T}}
\newcommand{\E}{\mathcal{E}}

\newcommand{\eq}[1]{\begin{align}#1\end{align}}
\newcommand{\eqn}[1]{\begin{align*}#1\end{align*}}

\title{superconvergence analysis of linear FEM based on the polynomial preserving recovery and Richardson extrapolation for\\ Helmholtz equation with high wave number}
\author{
Yu Du\footnotemark[1]
\thanks{Beijing computational science research center, Beijing, 100193, China. {\tt duyu87@csrc.ac.cn, dynju@qq.com}.
This research work is supported by a Tianhe--2JK computing time award at the Beijing Computational Science Research Center (CSRC).
The research of this author was supported in part by the China Postdoctoral Science Foundation under grant 2016M591053
and the National Natural Science Foundation of China under grants 11601026}
 \and
Haijun Wu\footnotemark[2]
\thanks{Department of Mathematics, Nanjing University, Jiangsu,
210093, China. {\tt hjw@nju.edu.cn}. This research was
partially supported by the National Natural Science Foundation of China under grants 11525103 and 91130004.}
\and
Zhimin Zhang\footnotemark[3]
\thanks{Beijing Computational Science Research Center, Beijing, 100193 and Department of Mathematics, Wayne State University, Detroit, MI 48202. {\tt zmzhang@csrc.ac.cn, zzhang@math.wayne.edu}. The research of this author was supported in part by the National Natural Science Foundation of China under grants 11471031, 91430216, U1530401, and the U.S. National Science Foundation through grant DMS--1419040.}
}

\begin{document}

\maketitle

\vspace{-1.4in}
\slugger{sinum}{200x}{xx}{x}{xxx--xxx}
\vspace{1.4in}

\setcounter{page}{1}

\begin{abstract} We study superconvergence property
of the linear finite element method with the polynomial preserving recovery (PPR)
and Richardson extrapolation for the two dimensional Helmholtz equation. The $H^1$-error estimate
with explicit dependence on the wave number $k$ {is} derived.
 First, we prove that under the assumption $k(kh)^2\leq C_0$ ($h$ is the mesh size) and certain mesh condition,
 the estimate between the finite element solution and the linear interpolation of
the exact solution is superconvergent under the $H^1$-seminorm, although the pollution error still exists.
Second, we prove a similar result for the recovered gradient by PPR and
found that the PPR can only improve the interpolation error and has no effect on the pollution error.
Furthermore, we estimate the error between the finite element gradient and recovered gradient and discovered that
the pollution error is canceled between these two quantities.
Finally, we apply the Richardson extrapolation to recovered gradient and demonstrate
numerically that PPR combined with the Richardson extrapolation can reduce the interpolation and pollution
errors simultaneously, and therefore, leads to an asymptotically exact {\it a posteriori} error estimator.
All theoretical findings are verified by numerical tests.
\end{abstract}



\setcounter{page}{1}

\section{Introduction}\label{intro}
Let $\Om\in\R^2$ be a bounded polygon with boundary $\Ga:=\pa\Om$. We consider the Helmholtz problem:
\eq{
-\De u - k^2 u &=f  \qquad\mbox{in  } \Om,\label{eq1.1a}\\
\frac{\pa u}{\pa n} +\i k u &=g \qquad\mbox{on } \Ga,\label{eq1.1b}
}
where $\i=\sqrt{-1}$ denotes the imaginary unit and $n$ denotes the unit outward normal to $\Ga$.
The above Helmholtz problem is an approximation of the
following acoustic scattering problem (with time dependence $e^{\i\om t}$):
\eq{ -\De u-k^2u &= f \quad \mathrm{in}\ \R^2,\label{eq1.2a}\\
\sqrt{r}\bigg(\frac{\pa(u-u^{inc})}{\pa r}+\i k(u-u^{inc})\bigg)&\rightarrow0 \quad \mathrm{as}\ r=\abs{x}\rightarrow\infty, \label{eq1.2b}   }
where $u^{inc}$ is the incident wave and $k$ is known as the wave number. The Robin boundary condition
\eqref{eq1.1b} is known as the first order approximation of the radiation condition \eqref{eq1.2b} (cf. \cite{em79}).
We remark that the Helmholtz problem \eqref{eq1.1a}--\eqref{eq1.1b} also arises in applications as a consequence of frequency domain
treatment of attenuated scalar waves (cf. \cite{dss94}).

We know that the finite element method of fixed order for the Helmholtz
problem \eqref{eq1.1a}--\eqref{eq1.1b} at high frequencies ($k\gg1$) is subject to the effect of pollution:
the ratio of the error of the finite element solution to the error of the best approximation
from the finite element space cannot be uniformly bounded with respect to $k$
\cite{Ainsworth04,bs00,bips95,dbb99,harari97,ib95a,ib97}.
More precisely, the linear finite element method for a $2$-D Helmholtz problem
satisfies the following error estimate under the mesh constraint $k(kh)^2\leq C_0$ \cite{zw,dw}:
\eq{ \norm{\na(u-u_h)}_{L^2(\Om)} \leq C_1kh+C_2k(kh)^2. \label{int-eq1}}
Here $u_h$ is the linear finite element solution, $h$ is the mesh size and $C_i,i=1,2$ are positive constants independent of $k$ and $h$.
It is easy to see that the order of the first term on the right hand side of \eqref{int-eq1} is the same to that
of the interpolation error in $H^1$-seminorm and it can dominate the error bound only if $k(kh)$ is small
enough. However, the second term on the {right-hand side} of \eqref{int-eq1} dominates the estimate under other mesh conditions.
For example, $kh$ is fixed and $k$ is large enough. The term $C_2k(kh)^2$ is called the pollution error of
the finite element solution.

Considerable efforts have been made in analysis of different numerical methods
for the Helmholtz problem with large wave number in the literature.
The readers are referred to \cite{ak79,dss94,sch74} for asymptotic error estimates of general DG methods and
\cite{ib95a,ib97} for pre-asymptotic error estimates of a one-dimensional problem discretized on equidistant grid.
For more pre-asymptotic error estimates, Please refer to \cite{ms10,ms11}
and \cite{zbw,zw} for classical finite element methods as well as interior penalty finite element methods.
For other methods solving the Helmholtz problems, such as the interior
penalty discontinuous Galerkin method or the source transfer domain decomposition method,
one can read \cite{mps13,fw09,fw11,zd,dzh,cx}.

In this work, we investigate the superconvergence property of the linear finite element method
when being post-processed by the polynomial preserving recovery (PPR) for the Helmholtz problem.
PPR was proposed by Zhang and Naga \cite{zn05} in 2004 and has been successfully applied to finite element methods.
COMSOL Multiphysics adopted PPR as a post-processing tool since 2008 \cite{comsol}.
One important feature of PPR is its superconvergence property for the recovered gradient.
To learn more about PPR, readers are referred to \cite{z04,z04t,nz04,wz07}.
Some theoretical results about recovery techniques and recovery-type error
estimators can be found in \cite{bx03,lmw,zl99,xz03,yz01}.

Let $V_h$ be the linear finite element space and denote $G_h:V_h\rightarrow V_h\times V_h$ as the gradient recovery operator from PPR.
We obtain the following estimate:
\eq{ \norm{\na u-G_hu_h}_{L^2(\Om)} \ls kh^{1+\al} + k(kh)^2, \qquad (0<\alpha\le 1) \label{eq_sec_gh}}
under the mesh condition $k(kh)^2\leq C$, where $C$ is a constant independent of $k$ and $h$. Furthermore, we prove
\eqn{ \norm{G_hu_h-\na u_h}_{L^2(\Om)} \ls kh + k(kh)^3, }
which means that $\norm{G_hu_h-\na u_h}_{L^2(\Om)}$, i.e., using PPR alone, can not measure the $H^1$-error of the numerical solution well. However, the super-convergence $O(h^2)$ of the recovered gradient with $\al=1$ makes it possible to apply
the Richardson extrapolation on the recovered gradient of the numerical solution. We show the asymptotic exactness of the a posteriori error estimator $\norm{RG_hu_h-\na u_h}_{L^2(\Om)}$ by numerical tests, where $R\cdot$ is the Richardson extrapolation operator.

The remainder of this paper is organized as follows: some notations, FEM and the mesh constraints are
introduced in section \ref{pre}. In section \ref{sup}, we prove the superconvergence between the interpolant and
the finite element solution to the problem with Robin boundary \eqref{eq1.1a}--\eqref{eq1.1b}. In section \ref{gra},
we prove the superconvergence property of $G_h$ in the Sobolev space $H^3$ and show the most important result,
that is the error estimate of $G_hu_h$. Then we try to give the reason for the effect of $G_h$ to the pollution error in
section \ref{est}. Finally, we simulate a model problem by the linear FEM, PPR method and the Richardson extrapolation in section \ref{num}.
It is shown that the recovered gradient can be improved by the Richardson extrapolation further and the a posterior error estimator based on the PPR
and Richardson extrapolation is exact asymptotically.

Throughout the paper, $C$ is used to denote a generic positive constant which is
independent of $h, k, f$ and $g$. We also use the shorthand
notation $A\ls B$ and $A\gtrsim B$ for the inequality $A\leq CB$ and $A\geq B$.
$A\eqsim B$ is a shorthand notation for the statement $A\ls B$ and $B\ls A$.
We assume that $k\gg1$ since we are considering high-frequency problems and that $k$ is constant
on $\Om$ for ease of presentation. We also assume that $\Om$ is a strictly star-shaped domain.
Here ``strictly star-shaped'' means that there exist a point $x_\Om\in\Om$ and a positive
constant $c_\Om$ depending only on $\Om$ such that
\eqn{ (x-x_\Om)\cdot n\geq c_\Om\quad \forall x\in\Ga. }

\section{Preliminaries} \label{pre}
We first introduce some notation. The standard Sobolev and Hilbert space, norm, and inner product notation
are adopted. Their definitions can be found in \cite{bs08,ciarlet78}. In particular, $(\cdot,\cdot)_Q$ and $\pd{\cdot,\cdot}_\Sigma$
for $\Sigma=\pa Q$ denote the $L^2$-inner product on complex-valued $L^2(Q)$ and $L^2(\Sigma)$ spaces, respectively.
For simplicity, we denote $(\cdot,\cdot):=(\cdot,\cdot)_\Om$, $\pd{\cdot,\cdot}:=\pd{\cdot,\cdot}_{\pa\Om}$,
$\norm{\cdot}_j:=\norm{\cdot}_{H^j(\Om)}$, and $\abs{\cdot}_j:=\abs{\cdot}_{H^j(\Om)}$.

Let $\T_h$ be a regular triangulation of the domain $\Om$, $\E_h$ be the set of all edges of $\T_h$ and
$\N_h$ be the set of all nodal points. For any $\tau\in\T_h$, we denote by $h_\tau$ its diameter and by $\abs{\tau}$ its area.
Similarly, for each edge $e\in\E_h$, define $h_e:={\rm diam}(e)$. Let $h=\max_{\tau\in\T_h}h_\tau$. Assume that $h_\tau\eqsim h$.
We denote all the boundary edges by $\E_h^B:=\set{e\in\E_h:e\subset\Ga}$ and the interior edges by
$\E_h^I:=\E_h\backslash\E_h^B$.

Let $V_h$ be the approximation space of continuous piecewise linear polynomials,
 that is,
\eqn{ V_h:=\set{v_h\in H^1(\Om): v_h|_\tau\in P_1(\tau)\ \forall \tau\in\T_h}, }
where $P_1(\tau)$ denotes the set of all polynomials defined on $\tau$ with degree $\leq1$.

Denote by $a(u,v)=(\na u,\na v)\ \forall u,v\in H^1(\Om)$. The variational problem to \eqref{eq1.1a}--\eqref{eq1.1b} reads as follows:
Find $u\in H^1(\Om)$ such that
\eq{ a(u,v)-k^2(u,v)+\i k\pd{u,v}=(f,v)+\pd{g,v}\quad \forall v\in H^1(\Om). \label{pre-var} }
Then the linear finite element solution $u_h\in V_h$ satisfies
\eq{ a(u_h,v_h)-k^2(u_h,v_h)+\i k\pd{u_h,v_h}=(f,v_h)+\pd{g,v_h}\quad \forall v_h\in V_h. \label{fem} }

Throughout the paper, we assume that the data $f$ is sufficiently smooth and
$g\in H^2(\Ga)$ such that $u\in H^3(\Om)$.
Denote by
\eq{C_{u,g}=\sum_{j=1}^3 k^{-(j-1)}\norm{u}_j + \sum_{j=1}^2 k^{-j}\abs{g}_{H^j(\Ga)}.}

We remark that in recent years there have been some superconvergence results for recovered gradients \cite{wz07,xz03,yz01}.
All of them assumed at least $u\in H^3(\Om)\cap W^2_\infty(\Om)$ instead of $u\in H^3(\Om)$.
The function $C_{u,g}$ could be treated as a constant in this paper since $\norm{u}_j$
is bounded by $\max(k^0,k^{j-1})$. The reader is referred to \cite{mps13,ms10,ms11} for the estimates of $u$.

The following norm on $H^1(\Om)$ is useful for the subsequent analysis:
\eq{ \norme{v}:=\big(\norm{\na v}_0^2+k^2\norm{v}_0^2 \big)^{\frac{1}{2}}\quad \forall v\in H^1(\Om).\label{norme} }
%
%
The following lemma is proved in \cite{zw,dw}.
\begin{lemma} \label{lemma1}
For $u$ and $u_h$, the solutions to \eqref{eq1.1a}--\eqref{eq1.1b} and \eqref{fem}, there exists a
constant $C_0$ independent of $k$ and $h$ such that if $k(kh)^2\leq C_0$, then the following error estimates hold:
\eqn{ \norm{u-u_h}_1 &\ls \big(kh+k(kh)^2\big) \frac{\abs{u}_2}{k},\\
k\norm{u-u_h}_0 &\ls \big((kh)^2+k(kh)^2\big) \frac{\abs{u}_2}{k}. }
\end{lemma}
Note that $k^{-1}\abs{u}_2\le C_{u,g}$.

We begin with some definitions regarding meshes.
For an interior edge $e\in\E_h^I$, we denote $\Om_e=\tau_e\cup\tau_e'$,
a patch formed by the two elements $\tau_e$ and $\tau_e'$ sharing $e$, see Figures~\ref{figNot}-\ref{figNotb}.
For any edge $e\in\E_h$ and an element $\tau$ {with} $e\subset\tau$,
$\ta_e$ denotes the angle opposite of the edge $e$ in $\tau${,
$\bt_e$ denotes the unit tangent vector of $e$ with counterclockwise orientation and $\bn_e$, the unit outward normal vector of $e$,
$h_e, h_{e+1}$, and $h_{e-1}$ denote the lengths of the three edges of $\tau$, respectively.
Here the subscript $e+1$ or $e-1$ is for orientation.
Note that all triangles in the triangulation are orientated counterclockwise,
and the} index $'$ is added for the corresponding quantities in $\tau'$ with $\bt_e=-\bt_e'$ and $\bn_e=-\bn_e'$
due to the orientation.

For any $e\in\E_h^I$ (cf. Figure~\ref{figNot}), we say that $\Om_e$ is an $\ep$ approximate parallelogram if the
lengths of any two opposite edges differ by at most $\ep$, that is,
\eqn{ \abs{h_{e-1}-h_{e-1}'} + \abs{h_{e+1}-h_{e+1}'}\le\ep. }

For any $e\in\E_h^B$ (cf. Figure~\ref{figNotb}), we say that $\tau_e$ is an $\ep$ approximate isosceles triangle if the lengths of
its two edges $e-1$ and $e+1$ differ by at most $\ep$, that is,
\eqn{ \abs{h_{e+1}-h_{e-1}}\le\ep. }

\begin{definition} \label{meshcond}
The triangulation $\T_h$ is said to satisfy \emph{$\al$ approximation condition} if there exists a constant $\al\geq0$
such that
\begin{itemize}
  \item[(a)] the patch $\Om_e$ is an $O(h^{1+\al})$ approximate parallelogram for any interior edge $e\in\E_h^I$;
  \item[(b)] the triangle $\tau_e$ is an $O(h^{1+\al})$ approximate isosceles triangle for any boundary edge $e\in\E_h^B$;
\end{itemize}
\end{definition}

\emph{Remark} 2.1.
For interior edges, the restriction ``\emph{$h^{1+\al}$ approximate parallelogram}'' is often used
 to prove the superconvergence property for problems with the Dirichlet boundary condition \cite{cx07,wz07},
 when boundary edges $\E_h^B$ can be ignored since $u_h-u_I\equiv0$ where
$u_I$ is the linear interpolant of $u$. However, ignoring the edges in $\E_h^B$ is impossible for the Robin condition
\eqref{eq1.1b}. As a result, more restrictions are put on the boundary edges.
Note that this restriction is technique and just for theoretical
purpose. In fact, one can still get results of superconvergence under general meshes which do not satisfy the condition, such as Chevron pattern uniform mesh.

\begin{figure}
\begin{center}
\setlength{\unitlength}{0.25mm}
  \begin{picture}(390,110)

  \linethickness{0.25mm}
  \put(110,10){\line(1,0){150}}
  \put(160,110){\line(1,0){150}}
  \put(110,10){\line(1,2){50}}
  \put(260,10){\line(1,2){50}}
  \put(260,10){\line(-1,1){100}}

  \put(185,60){$e$}
  \put(150,40){$\tau$}
  \put(260,80){$\tau'$}
  \put(100,60){$e+1$}
  \put(185,0){$e-1$}
  \put(120,15){$\ta_e$}
  \put(285,95){$\ta_e'$}

  \put(210,60){\vector(-1,-1){25}}
  \put(210,60){\vector(1,1){25}}
  \put(170,35){$\bn_e'$}
  \put(235,85){$\bn_e$}
  \end{picture}
\end{center}
\caption{Notation in the patch $\Om_e$.}
\label{figNot}
\end{figure}
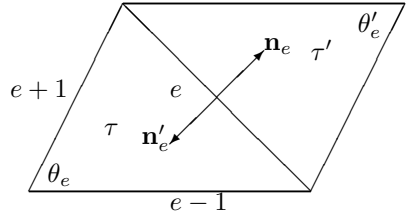

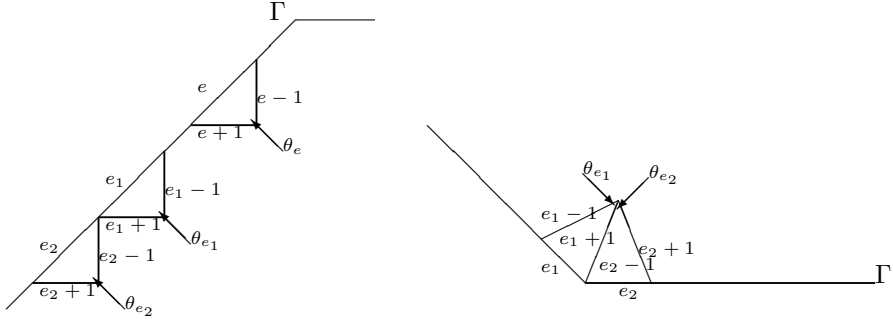
\begin{figure}
\begin{center}
\setlength{\unitlength}{0.35mm}
  \begin{picture}(390,110)

  \put(120,110){\line(1,0){30}}
  \put(10,0){\line(1,1){110}}
  \put(110,110){$\Ga$}

  \linethickness{0.25mm}
  \put(20,10){\line(1,0){25}}
  \put(45,10){\line(0,1){25}}
  \put(22.5,22.5){\scriptsize $e_2$}
  \put(45,18.5){\scriptsize $e_2-1$}
  \put(22.5,5){\scriptsize $e_2+1$}
  \put(55,0){\vector(-1,1){12}}
  \put(55,0){\scriptsize $\ta_{e_2}$}


  \put(45,35){\line(1,0){25}}
  \put(70,35){\line(0,1){25}}
  \put(47.5,47.5){\scriptsize $e_1$}
  \put(70,43.5){\scriptsize $e_1-1$}
  \put(47.5,30){\scriptsize $e_1+1$}
  \put(80,25){\vector(-1,1){12}}
  \put(80,25){\scriptsize $\ta_{e_1}$}

  \put(80,70){\line(1,0){25}}
  \put(105,70){\line(0,1){25}}
  \put(82.5,82.5){\scriptsize $e$}
  \put(105,78.5){\scriptsize $e-1$}
  \put(82.5,65){\scriptsize $e+1$}
  \put(115,60){\vector(-1,1){12}}
  \put(115,60){\scriptsize $\ta_{e}$}

  \put(230,10){\line(1,0){110}}
  \put(230,10){\line(-1,1){60}}
  \put(340,10){$\Ga$}

  \put(230,10){\line(2,5){12.5}}
  \put(255,10){\line(-2,5){12.5}}
  \put(242.5,5){\scriptsize $e_2$}
  \put(250,21){\scriptsize $e_2+1$}
  \put(235,15){\scriptsize $e_2-1$}
  \put(254,50){\vector(-1,-1){12}}
  \put(254,50){\scriptsize $\ta_{e_2}$}

  \put(213.3,26.7){\line(2,1){29}}
  \put(213,14.3){\scriptsize $e_1$}
  \put(213,34.3){\scriptsize $e_1-1$}
  \put(220,25){\scriptsize $e_1+1$}
  \put(229,51.5){\vector(1,-1){12}}
  \put(229,51.5){\scriptsize $\ta_{e_1}$}

  \end{picture}
\end{center}
\caption{Notation in the boundary elements.}
\label{figNotb}
\end{figure}

\section{ Superconvergence between the finite element solution and linear interpolant} \label{sup}
Different from most other investigations in the literatures where the Dirichlet boundary condition is assumed, we
consider the superconvergence between the FE solution $u_h$ under the Robin boundary condition
and the linear interpolant $u_I$ of the exact solution $u$.
Since $u_h$ may not equal $u_I$ on the boundary $\Ga$, some more strict mesh conditions and special arguments are needed to
establish the desired superconvergence result.

First we introduce a quadratic interpolant $\psi_Q=\Pi_Q\psi$ of $\psi$ based on nodal values and
moment conditions on edges,
\eq{ (\Pi_Q\phi)(z)=\phi(z),\quad \int_e\Pi_Q\phi=\int_e\phi\quad \forall z\in\N_h,e\in\E_h. }

The following fundamental identity for $v_h\in P_1(\tau)$ has been proved in \cite{cx07}:
\eq{ \int_\tau \na(\phi-\phi_I)\cdot\na v_h = \sum_{e\in\pa\tau} \bigg( \be_e\int_e\frac{\pa^2\phi_Q}{\pa \bt_e^2}\frac{\pa v_h}{\pa \bt_e} + \ga_e\int_e\frac{\pa^2\phi_Q}{\pa \bt_e\pa \bn_e}\frac{\pa v_h}{\pa \bt_e} \bigg) \label{fdmtliequ} }
where
\eq{ \be_e=\frac{1}{12}\cot\ta_e(h_{e+1}^2-h_{e-1}^2),\quad \ga_e=\frac{1}{3}\cot\ta_e\abs{\tau}, }
and $\phi_I\in P_1(\tau)$ is the linear interpolant of $\phi$ on $\tau$. The following lemma can be easily
obtained \cite{cx07,wz07}.
\begin{lemma} \label{parsestm}
We denote $\bm_e$ by $\bt_e$ or $\bn_e$. Assume that $\T_h$ satisfies the \emph{$\al$ approximation condition}, then
we have the following estimates:
\begin{itemize}
  \item[(a)] For any interior edge $e\in\E_h^I$,
  \eq{ &\abs{\be_e}+\abs{\be_e'}\ls h^2,\quad \abs{\ga_e}+\abs{\ga_e'}\ls h^{2};\label{beta_gamma1}\\
  &\abs{\be_e-\be_e'}\ls h^{2+\al},\quad \abs{\ga_e-\ga_e'}\ls h^{2+\al}.\label{beta_gamma2}}
  \item[(b)] For two adjacent edges $e_1,e_2\in\E_h^B$, that is $e_1\cap e_2\neq\emptyset$,
  \eq{ &\abs{\be_{e_1}}+\abs{\be_{e_2}}\ls h^{2+\al},\quad \abs{\ga_{e_1}}+\abs{\ga_{e_2}}\ls h^{2}\label{beta_gamma3}\\
  & \abs{\ga_{e_1}-\ga_{e_2}}\ls h^{2+\al}\label{beta_gamma4}}
  \item[(c)] For any edge $e\in\E_h$, $e\subset\pa\tau_e$,
  \eq{  &\int_e \frac{\pa^2\phi}{\pa \bt_e\pa \bm_e}\frac{\pa v_h}{\pa \bt_e} \ls (\norm{\phi}_{H^{3}(\tau_e)}+h^{-1}\norm{\phi}_{H^{2}(\tau_e)})\norm{\na v_h}_{L^2(\tau_e)};\label{trcinq1}\\
  &\int_e \frac{\pa^2(\phi-\phi_Q)}{\pa \bt_e\pa \bm_e}\frac{\pa v_h}{\pa \bt_e} \ls \abs{\phi}_{H^{3}(\tau_e)}\norm{\na v_h}_{L^2(\tau_e).\label{trcinq2}}
}
\end{itemize}
\end{lemma}
\begin{proof}
The inequalities \eqref{beta_gamma1}--\eqref{beta_gamma3} follow from the \emph{$\al$ approximation condition}. From the condition (a) and (b) in Definition~\ref{meshcond}, we have
for any $e_1, e_2\in\E_h^B$ satisfying $e_1\cap e_2\neq\emptyset$ (cf. Figure~\ref{figNotb}),
\eqn{ \frac{\big|h_{e_1-1}h_{e_1+1}\cos\ta_{e_1}-h_{e_2-1}h_{e_2+1}\cos\ta_{e_2}\big|}{h} \ls h^{1+\al}, }
which implies \eqref{beta_gamma4}.

Finally, the inequalities \eqref{trcinq1} and \eqref{trcinq2} follow from the trace theorem.
\end{proof}

\begin{lemma} \label{uui}
Assume that $\T_h$ satisfies {the} \emph{$\al$ approximation ondition}. Then for any $v_h\in V_h$,
\eq{ \abs{\int_\Om\na(u-u_I)\cdot\na v_h}\ls \left((kh)^2+kh^{1+\al}\right)\norme{\na v_h}_{L^2(\Om)} C_{u,g}. }
Here $u_I$ is the linear interpolant of $u$ on $\Om$.
\end{lemma}
\begin{proof}
From \eqref{fdmtliequ}, we have
\eqn{
\int_\Om\na(u-u_I)\cdot\na v_h &= \sum_{\tau\in\T_h}\sum_{e\subset\pa\tau} \left( \be_e\int_e\frac{\pa^2 u_Q}{\pa \bt_e^2}\frac{\pa v_h}{\pa \bt_e} + \ga_e\int_e\frac{\pa^2 u_Q}{\pa \bt_e \pa \bn_e}\frac{\pa v_h}{\pa \bt_e} \right)\\
&= I_1+I_2,
}
where
\eqn{
&I_1 = \sum_{e\in\E_h^I} \left[ (\be_e-\be_e')\int_e\frac{\pa^2 u}{\pa t_e^2}\frac{\pa v_h}{\pa t_e} + (\ga_e-\ga_e')\int_e\frac{\pa^2 u}{\pa t_e \pa n_e}\frac{\pa v_h}{\pa t_e} \right. \\
&\qquad\qquad + \be_e\int_e\frac{\pa^2 (u_Q-u)}{\pa t_e^2}\frac{\pa v_h}{\pa t_e} + \ga_e \int_e\frac{\pa^2 (u_Q-u)}{\pa t_e \pa n_e}\frac{\pa v_h}{\pa t_e}\\
&\qquad\qquad \left. + \be_e'\int_e\frac{\pa^2 (u-u_Q)}{\pa t_e^2}\frac{\pa v_h}{\pa t_e} + \ga_e' \int_e\frac{\pa^2 (u-u_Q)}{\pa t_e \pa n_e}\frac{\pa v_h}{\pa t_e} \right],\\
&I_2 = \sum_{e\in\E_h^B} \left[ \be_e\int_e\frac{\pa^2 u}{\pa t_e^2}\frac{\pa v_h}{\pa t_e} + \ga_e\int_e\frac{\pa^2 u}{\pa t_e \pa n_e}\frac{\pa v_h}{\pa t_e} \right. \\
&\qquad\qquad \left. + \be_e\int_e\frac{\pa^2 (u_Q-u)}{\pa t_e^2}\frac{\pa v_h}{\pa t_e} + \ga_e \int_e\frac{\pa^2 (u_Q-u)}{\pa t_e \pa n_e}\frac{\pa v_h}{\pa t_e}\right].
}

First, $I_1$ can be estimated by Lemma~\ref{parsestm} and H\"{o}lder's inequality:
\eq{
\abs{I_1} &\ls \sum_{e\in\E_h^I} \left( (h^{2+\al}+h^2)\norm{u}_{H^{3}(\tau_e)} + h^{1+\al}\norm{u}_{H^{2}(\tau_e)} \right) \norm{\na v_h}_{L^2(\tau_e)} \label{I1est}\\
&\ls \left( (h^{2+\al}+h^2)\norm{u}_3 + h^{1+\al}\norm{u}_2 \right) \norm{\na v_h}_0\nn \\
&\ls \left( (kh)^{2} + kh^{1+\al} \right)\norm{\na v_h}_0  C_{u,g}.\nn
}

Next we estimate $I_2$. From \eqref{beta_gamma3} and \eqref{trcinq2},
\eq{
I_{2,1} &:= \sum_{e\in\E_h^B} \left[ \be_e\int_e\frac{\pa^2 u}{\pa t_e^2}\frac{\pa v_h}{\pa t_e} + \be_e\int_e\frac{\pa^2 (u_Q-u)}{\pa t_e^2}\frac{\pa v_h}{\pa t_e} + \ga_e \int_e\frac{\pa^2 (u_Q-u)}{\pa t_e \pa n_e}\frac{\pa v_h}{\pa t_e}\right]\label{part1_I2}\\
&\ls \sum_{e\in\E_h^B} \big( h^{1+\al}\norm{u}_{H^2(\tau_e)}+(h^{2+\al}+h^2)\norm{u}_{H^3(\tau_e)} \big) \norm{\na v_h}_{L^2(\tau_e)}\nn\\
&\ls \big( kh^{1+\al}+(kh)^{2} \big)\norm{\na v_h}_0  C_{u,g}.\nn
}
We turn to the estimate of the remaining terms of $I_2$. Denote by $z_i$ the nodes on $\Ga$. Let $e_1$ and $e_2$ be two boundary edges in $\E_h^B$ sharing $z_i$ with counterclockwise orientation (cf. Figure~\ref{figNotb}). Denote by $\jm{\ga_e}_{z_i}=\ga_{e_2}-\ga_{e_1}$ and by $\N_h^v$ the set of vertices of the domain $\Om$. Then we have
\eq{
\sum_{e\in\E_h^B}\ga_e\int_e\frac{\pa^2 u}{\pa t_e \pa n_e}\frac{\pa v_h}{\pa t_e} &=  -  \sum_{e\in\E_h^B}\ga_e\int_e\frac{\pa^3 u}{\pa t_e^2 \pa n_e}v_h + \sum_{z_i\in\Ga\bigcap\N_h\setminus\N_h^v} \jm{\ga_e}_{z_i} \frac{\pa^2 u}{\pa t_e \pa n_e}(z_i) v_h(z_i) \label{part20_I2} \\
& + \sum_{z_i\in\N_h^v} \bigg( \ga_{e_2} \frac{\pa^2 u}{\pa t_{e_2} \pa n_{e_2}}(z_i) v_h(z_i) - \ga_{e_1} \frac{\pa^2 u}{\pa t_{e_1} \pa n_{e_1}}(z_i) v_h(z_i) \bigg)\nn\\
&:= I_{2,2} + I_{2,3} + I_{2,4}.\nn
}

It is easy to get
\eq{
I_{2,2} &= \sum_{e\in\E_h^B}\ga_e\int_e\frac{\pa^3 u}{\pa t_e^2 \pa n_e}v_h \ls h^2 \abs{\frac{\pa u}{\pa n}}_{H^2(\Ga)} \norm{v_h}_{L^2(\Ga)}\label{part22_I2}\\
& \ls h^2 \big(  \abs{g}_{H^2(\Ga)} + k \abs{u}_{H^2(\Ga)} \big) \cdot \norm{v_h}_0^{1/2} \norm{v_h}_1^{1/2} \nn\\
& \ls k^{3/2} h^2  \norme{v_h} C_{u,g}. \nn
}

Suppose that $w\in H^1([a,b])$ and denote by $h_{ab}=b-a$, we have
\eqn{
w^2(b) &= \int_a^b \bigg( \frac{x-a}{b-a} w^2(x) \bigg)' \dx = \frac{1}{b-a} \int_a^b w^2 + 2\int_a^b \frac{x-a}{b-a}ww' \\
& \leq \frac{1}{h_{ab}} \norm{w}_{L^2([a,b])}^2 + 2\abs{w}_{H^1([a,b])} \norm{w}_{L^2([a,b])},
}
which implies
\eq{
I_{2,3}&\leq  \sum_{z_i\in\Ga\bigcap\N_h} \abs{\jm{\ga_e}_{z_i}} \bigg( \frac{1}{h_{e_i}} \abs{\frac{\pa u}{\pa n_{e_i}}}_{H^1(e_i)}^2 + 2\abs{\frac{\pa u}{\pa n_{e_i}}}_{H^2(e_i)}\abs{\frac{\pa u}{\pa n_{e_i}}}_{H^1(e_i)}\bigg)^{1/2} \label{part21_I2}\\
& \cdot \bigg( \frac{1}{h_{e_i}} \norm{v_h}_{L^2(e_i)}^2 + 2\abs{v_h}_{H^1(e_i)}\norm{v_h}_{L^2(e_i)} \bigg)^{1/2}\nn\\
&\ls \max_{z_i\in\Ga\bigcap\N_h} \abs{\jm{\ga_e}_{z_i}} \bigg( \frac{1}{h} \abs{\frac{\pa u}{\pa n}}_{H^1(\Ga)}^2 + \abs{\frac{\pa u}{\pa n}}_{H^2(\Ga)}\abs{\frac{\pa u}{\pa n}}_{H^1(\Ga)}\bigg)^{1/2}\nn\\
&\cdot \bigg( \frac{1}{h} \norm{v_h}_{L^2(\Ga)}^2 + \abs{v_h}_{H^1(\Ga)}\norm{v_h}_{L^2(\Ga)} \bigg)^{1/2}\nn\\
&\ls \frac{\max_{z_i\in\Ga\bigcap\N_h}\abs{\jm{\ga_e}_{z_i}}}{h} \bigg( \big(\abs{g}_{H^1(\Ga)}^2 + k\abs{u}_{H^1(\Ga)}^2\big) + h\big(\abs{g}_{H^2(\Ga)}+k\abs{u}_{H^2(\Ga)}\big) \cdot \nn\\
&\quad\big(\abs{g}_{H^1(\Ga)}+k\abs{u}_{H^1(\Ga)}\big) \bigg)^{1/2} \cdot \bigg( \norm{v_h}_{L^2(\Ga)}^2 + h \abs{v_h}_{H^1(\Ga)}\norm{v_h}_{L^2(\Ga)} \bigg)^{1/2}\nn\\
&\ls \frac{\max_{z_i\in\Ga\bigcap\N_h}\abs{\jm{\ga_e}_{z_i}}}{h} k^{3/2} \bigg( \norm{v_h}_0 \norm{v_h}_1 + h^{1/2} \norm{v_h}_1^{3/2}\norm{v_h}_0^{1/2}\bigg)^{1/2} C_{u,g}\nn\\
&\ls k \frac{\max_{z_i\in\Ga\bigcap\N_h}\abs{\jm{\ga_e}_{z_i}}}{h}  \norme{v_h} C_{u,g} \ls kh^{1+\al} \norme{v_h} C_{u,g},\nn
}
where we have used the second inequality in \eqref{beta_gamma3}.\

Since the number of the vertices of $\Om$ is a bounded constant independent of all the parameters $k$, $h$ and $\al$, from the trace theorem we have
\eq{
I_{2,4} &\ls \max_{z_i\in\N_h^v} \abs{\ga_e} \norm{\frac{\pa^2 u}{\pa t_{e} \pa n_{e}}}_{H^1(\Ga)}^{1/2} \norm{\frac{\pa^2 u}{\pa t_{e} \pa n_{e}}}_{L^2(\Ga)}^{1/2} \cdot \norm{v_h}_{H^{1/2}(\Ga)} \label{eq_I24}\\
& \ls h^2 \bigg(\norm{g}_{H^2(\Ga)}+k\norm{u}_{H^2(\Ga)}\bigg)^{1/2}\bigg(\norm{g}_{H^1(\Ga)}+k\norm{u}_{H^1(\Ga)}\bigg)^{1/2} \norm{v_h}_1 \nn\\
&\ls k^{3/2} h^2 \norme{v_h} C_{u,g}.\nn
}

Combining the inequalities \eqref{part1_I2}--\eqref{eq_I24} we have
\eq{ 
\abs{I_2} = \abs{I_{2,1}+I_{2,2}+I_{2,3}+I_{2,4}} \ls \big(kh^{1+\al}+(kh)^2\big) \norme{v_h} C_{u,g}.\label{eq_I2}
}

Finally, combining the estimates of $I_1$ and $I_2$ we complete the proof.
\end{proof}

Based on Lemma~\ref{uui}, we can obtain one of our main results in this paper.
\begin{theorem} \label{uhui}
Assume that $\T_h$ satisfies the \emph{$\al$ approximation condition}. There exists a constant $C_0$
independent of $k$ and $h$, such that if
\eq{k(kh)^2\leq C_0,}
we have
\eq{ \norme{u_h-u_I} \ls \big(kh^{1+\al}+k(kh)^2\big) C_{u,g}. \label{eq:th1eq1} }
\end{theorem}

\begin{proof}
For simplicity of presentation, we denote $v_h=u_h-u_I$. By the definition~\eqref{norme} and the Galerkin orthogonality, we have
\eq{
&\norme{u_h-u_I}^2 = \Re \left( a(u_h-u_I,v_h) + k^2(u_h-u_I,v_h) \right) \label{eq1}\\
&= \Re \left( a(u_h-u_I,v_h) - k^2(u_h-u_I,v_h) + \i k\pd{u_h-u_I,v_h} + 2k^2(u_h-u_I,v_h) \right)  \nn\\
&= \Re \left( a(u-u_I,v_h) - k^2(u-u_I,v_h) + \i k\pd{u-u_I,v_h} + 2k^2(u_h-u_I,v_h) \right).\nn
}
It is well known that
\eq{ k\norm{u-u_I}_{L^2(\Om)} \ls kh^2\norm{u}_{H^2(\Om)}.}
From Lemma~\ref{lemma1}, we know that there exists a constant $C_0$ such that if $k(kh)^2\leq C_0$,
the following inequality holds,
\eqn{ k\norm{u-u_h}_{L^2(\Om)} \ls \left((kh)^2+k(kh)^2\right) C_{u,g}. }
Then we have
\eq{2k^2(u_h-u_I,v_h)&\leq 2k\norm{u_h-u_I}_{L^2(\Om)}\cdot k\norm{v_h}_{L^2(\Om)} \\
&\leq 2\left(k\norm{u-u_h}_{L^2(\Om)}+k\norm{u-u_I}_{L^2(\Om)}\right)\cdot k\norm{v_h}_{L^2(\Om)}   \nn\\
&\ls  \left((kh)^2+k(kh)^2\right) \norme{v_h} C_{u,g}.\nn
}
On the other hand, by the trace inequality,
\eq{
\abs{k\pd{u-u_I,v_h}} &\leq k\norm{u-u_I}_{L^2(\pa\Om)}\norm{v_h}_{L^2(\pa\Om)} \label{eqend}   \\
&\ls  kh^2\norm{u}_{H^2(\pa\Om)}\norm{v_h}_{L^2(\pa\Om)}   \nn\\
&\ls k^{1/2}h^2\norm{u}^{1/2}_{H^2(\Om)}\norm{u}^{1/2}_{H^3(\Om)}\norme{v_h}    \nn\\
&\ls (kh)^2 \norme{v_h} C_{u,g}.\nn
}
Therefore, if $k(kh)^2\leq C_0$, by Lemma~\ref{uui} and \eqref{eq1}--\eqref{eqend}, we have
\eqn{
\norme{u_h-u_I}^2 &\leq \abs{a(u-u_I,v_h)} + \abs{k^2(u-u_I,v_h)}\nn\\
&\qquad + \abs{k\pd{u-u_I,v_h}} + \abs{2k^2(u_h-u_I,v_h)} \\
&\ls \left((kh)^2+kh^{1+\al}\right)\norme{v_h} C_{u,g} + (kh)^2 \norme{v_h} C_{u,g} \nn\\
&\qquad+ \left((kh)^2+k(kh)^2\right) \norme{v_h} C_{u,g} \nn\\
&\ls \left(kh^{1+\al}+(kh)^2+k(kh)^2\right)\norme{v_h} C_{u,g}. \nn\\
&\ls \left(kh^{1+\al}+k(kh)^2\right)\norme{v_h} C_{u,g}. \nn
}
This completes the proof.
\end{proof}

\section{The gradient recovery operator $G_h$ and its superconvergence} \label{gra}
In this section, we apply a gradient recovery operator developed in 2004, called polynomial preserving
recovery (PPR) \cite{nz04,z04,zn05}, to improve the finite element solution.
We first introduce the gradient recovery operator $G_h:C(\Om)\mapsto V_h\times V_h$.
Given a node $z\in\N_h$, we select $n\geq6$ sampling points $z_j\in\N_h$, $j=1,2,\cdots,n$, in an element
patch $\om_z$ containing $z$ ($z$ is one of $z_j$) and fit a polynomial of degree $2$, in the least
squares sense, with values of $u_h$ at those sampling points. First, we find $p_2\in P_2(\om_z)$ for some
$w\in C(\Om)$ such that
\eq{ \sum_{j=1}^n(p_2-w)^2(z_j)=\min_{q\in P_2}\sum_{j=1}^n(q-w)^2(z_j). }
Here $P_2(\om_z)$ is the well-known piecewise quadratic polynomial space defined on $\om_z$.
The recovery gradient at $z$ is then defined as
\eq{ G_hw(z)=(\na p_2)(z). }

For the linear element, the above least squares fitting procedure has a unique
solution as long as those $n$ sampling points are not on the same conic curve \cite{nz04}.

Now we show some properties of the gradient recovery operator $G_h$:
\begin{itemize}
  \item[(i)] $\norm{G_hv_h}_0\ls\norm{\na v_h}_0\quad \forall v_h\in V_h$.
  \item[(ii)] {For} any nodal point $z$, $(G_hp)(z)=\na p(z)$ if $p\in P_j(\om_z),\ j=1,2$.
  \item[(iii)] $G_hw=G_hI_h^jw\quad\forall w\in C(\Om),\ j=1,2$.
\end{itemize}
Here $I_h^jw(j=1,2)$ are the linear nodal value interpolant and quadratic nodal value interpolant of $w$, respectively. The reader is referred to \cite{nz04,wz07,z04,zn05} for more details of these properties.

From (i), we know that
\eq{ \norm{G_hu_h-\na u}_0 &\leq \norm{G_hu_h-G_hu_I} + \norm{G_hu_I-\na u}_0 \label{ghuh}\\
&\ls \norm{\na(u_h-u_I)}_0 + \norm{G_hu_I-\na u}_0.\nn}
Here $u_I$ is the linear interpolant of $u$. The estimate for the first term of the
right hand side of the inequality \eqref{ghuh} follows from Theorem~\ref{uhui}.
Next we will estimate the second term.

\begin{lemma}\label{ghphi}
For any element $\tau\in\T_h$ and any function $\phi\in H^3(\tilde\tau)$,
\eq{ \norm{G_h\phi_I-\na \phi}_{L^2(\tau)}\ls h^2\norm{\phi}_{H^3(\tilde\tau)}, }
where $\tilde\tau=\bigcup\set{\om_z:z\in\N_h\cap\tau}$ and $\phi_I$ is the linear interpolant of $\phi$.
\end{lemma}
\begin{proof}
By the property (iii),
\eq{ \norm{G_h\phi_I-\na \phi}_{L^2(\tau)} = \norm{G_h\phi-\na \phi}_{L^2(\tau)}=\norm{G_hI_h^2\phi-\na \phi}_{L^2(\tau)}. \label{Gh-eq0} }
For any $\eta\in P_{2}(\tilde\tau)$, from the property (ii) and the fact that $G_h\eta\in V_h\times V_h$ and
$\na\eta\in P_1(\tilde\tau)\times P_1(\tilde\tau)$, it is easy to get that $G_h\eta=\na\eta$ in $\tau$, which implies
\eq{ \norm{G_hI_h^2\phi-\na \phi}_{L^2(\tau)} &= \norm{G_h(I_h^2\phi-\eta)-\na (\phi-\eta)}_{L^2(\tau)} \label{Gh-eq1}\\
&\leq \norm{G_h(I_h^2\phi-\eta)}_{L^2(\tau)}+\norm{\na (\phi-\eta)}_{L^2(\tau)}.\nn }
By the definition and properties of $G_h$,
\eq{ \norm{G_h(I_h^2\phi-\eta)}_{L^2(\tau)} &\ls h \max_{z\in \N_h\cap\tau} \abs{G_h(I_h^2\phi-\eta)(z)} \ls
h \norm{\na(I_h^2\phi-\eta)}_{L^\infty(\tilde\tau)} \label{Gh-eq1t}\\
& \ls \norm{\na(I_h^2\phi-\eta)}_{L^2(\tilde\tau)} \ls  \norm{\na(I_h^2\phi-\phi)}_{L^2(\tilde\tau)} + \norm{\na(\phi-\eta)}_{L^2(\tilde\tau)}.\nn }
Then, from \eqref{Gh-eq0}--\eqref{Gh-eq1t} we have
\eq{\label{Gh-eq2} \norm{G_hI_h^2\phi-\na \phi}_{L^2(\tau)} \ls \inf_{\eta\in P_2(\tilde\tau)} \norm{\na (\phi-\eta)}_{L^2(\tilde\tau)} +
\norm{\na(I_h^2\phi-\phi)}_{L^2(\tilde\tau)}. }
By {the} Hilbert--Bramble lemma and the scaling argument,
\eq{\label{Gh-eq3} \inf_{\eta\in P_2(\tilde\tau)}\norm{\na (\phi-\eta)}_{L^2(\tilde\tau)} \ls h^2\norm{\phi}_{H^3(\tilde\tau)}, }
and from the approximation theory
\eq{\label{Gh-eq4} \norm{\na(I_h^2\phi-\phi)}_{L^2(\tilde\tau)}\ls h^2 \norm{\phi}_{H^3(\tilde\tau)}. }
The proof is completed by combining \eqref{Gh-eq0} {with} \eqref{Gh-eq2}--\eqref{Gh-eq4}
\end{proof}

The following theorem is devoted to the estimate of the second term of \eqref{ghuh}. In fact, it
shows the superconvergence property of $G_h$ in $H^3(\Om)$ space for the linear element.
\begin{theorem} \label{main0}
We have the following estimate:
\eq{ \norm{G_hu_I-\na u}_0\ls (kh)^2  C_{u,g}. \label{eq:th0eq1} }
\end{theorem}
\begin{proof} From Lemma~\ref{ghphi} we have
\eqn{ \norm{G_hu_I-\na u}_0 &= \left(\sum_{\tau\in\M_h}\norm{G_hu_I-\na u}^2_{L^2(\tau)}\right)^{1/2} \ls h^2 \left(\sum_{\tau\in\M_h}\norm{u}^2_{H^3(\tilde\tau)}\right)^{1/2}  \\
&\ls h^2 \norm{u}_3 \ls (kh)^2 C_{u,g}. }
\end{proof}

The following preasymptotic superconvergence estimate of the gradient recovery operator $G_h$
can be proved by combining \eqref{ghuh}, Theorem~\ref{uhui} and Theorem~\ref{main0}.
\begin{theorem} \label{main1}
Assume that $\T_h$ satisfies the \emph{$\al$ approximation condition}.
Let $u$ and $u_h$ be the solutions to \eqref{eq1.1a}--\eqref{eq1.1b} and \eqref{fem}, respectively.
Then there exists a constant $C_0$ independent of $k$ and $h$ such that if $k(kh)^2\leq C_0$,
\eq{ \norm{G_hu_h-\na u}_0\ls \big(kh^{1+\al}+k(kh)^2\big)  C_{u,g}. \label{ghpoll} }
\end{theorem}

\emph{Remark} 4.1. From~Lemma \ref{lemma1} we know that
\eq{ \norm{\na u_h-\na u}_0 \ls \big(kh+k(kh)^2\big) C_{u,g}. \label{fempoll} }
We can find that the pollution error of the finite element solution is $C_1k(kh)^2$ from
\eqref{fempoll} and that of the gradient recovery operator $G_h$ is $C_2k(kh)^2$ from \eqref{ghpoll}, where
$C_1$ and $C_2$ are two constants independent of $k$ and $h$.
It is interesting that the orders of these pollution errors are the same.
It seems that the gradient recovery operator does not reduce the pollution error based on our analysis.
Indeed, our numerical tests in section~\ref{num} indicate that the pollution error is the same with or without the gradient recovery.
In the next section, We will provide some explanation.

\section{The estimate between $G_hu_h$ and $\na u_h$} \label{est}
In this section, we devote estimating the norm $\norm{G_hu_h-\na u_h}_0$, which motivate us to
combine the PPR method and the Richardson extrapolation and define the a posteriori error estimator shown in the subsection~\ref{sec_richardson_posteriori}.
First, we define an elliptic projection $P_h:V\rightarrow V_h$:
find $P_hu\in V_h$ such that
\eq{ a(P_hu,v_h)+\i k\pd{P_hu,v_h}=a(u,v_h)+\i k\pd{u,v_h}\quad \forall v_h\in V_h.\label{ellpro}}
In other words, the elliptic projection $P_hu$ of $u$ is the finite element approximation
to the solution of the following (complex-valued) Poisson problem:
\eq{ -\De u &= F \quad \mathrm{in} \quad \Om,\label{eq2.2a}\\
\frac{\pa u}{\pa n} + \i ku &= g \quad \mathrm{on}\quad \Ga,\label{eq2.2b}  }
for some given function $F$ which are determined by $u$.
This kind of elliptic projection is often used to study some properties,
such as stability and convergence, of the FEM for the Helmholtz problem.  Readers are referred to \cite{zw,zd,dzh,dw}.

\begin{lemma} \label{ellemm}
Assume that $u$ is $H^2$-regular. $u_h^+$ is its elliptic projection defined by \eqref{ellpro}.
There hold the following estimates:
\eq{
\norme{u-P_hu} &\ls \inf_{v_h\in V_h} \norme{u-v_h}, \label{ellemm-eq1}\\
\norm{u-P_hu}_{L^2(\Om)} &\ls h\inf_{v_h\in V_h}\norme{u-v_h}.\label{ellemm-eq2}
}
\end{lemma}
\begin{proof}
From Lemma~3.5 in \cite{zw}, we know that
\eqn{
\norm{u-P_hu}_1 &\ls \inf_{v_h\in V_h} \big(\norm{u-v_h}_1^2+k\norm{u-v_h}_{L^2(\Ga)}^2\big)^{1/2},\\
\norm{u-P_hu}_0 &\ls h \inf_{v_h\in V_h} \big(\norm{u-v_h}_1^2+k\norm{u-v_h}_{L^2(\Ga)}^2\big)^{1/2}.
}
Then the estimates \eqref{ellemm-eq1}--\eqref{ellemm-eq2} follow from
\eqn{ k\norm{u-v_h}_{L^2(\Ga)}^2 &\ls k\norm{u-v_h}_0\norm{u-v_h}_1\\
&\ls k^2\norm{u-v_h}_0^2 + \norm{u-v_h}_1^2 \ls \norme{u-v_h}^2.
}
\end{proof}

\begin{lemma}\label{naphu}
Assume that $\T_h$ satisfies the \emph{$\al$ approximation condition} and $u$ is the exact solution to \eqref{eq1.1a}--\eqref{eq1.1b}.
$P_hu$ is its elliptic projection defined by \eqref{ellpro} and $u_I$ is its linear interpolation. We have
\eq{
\norm{\na P_hu-\na u_I}_0 \ls \big(kh^{1+\al}+(kh)^2\big) C_{u,g}. \label{eq:nabla-phu-ui}
}
\end{lemma}
\begin{proof}
Denote $v_h=P_hu-u_I$. By the Galerkin orthogonality,
\eq{ \norme{P_hu-u_I}^2 &\ls\Re\big( a(P_hu-u_I,v_h) + \i k\pd{P_hu-u_I,v_h} \big) + k^2(P_hu-u_I,v_h)\\
&\ls\Re\big( a(u-u_I,v_h) + \i k\pd{u-u_I,v_h} \big) + k^2(P_hu-u_I,v_h)\nn\\
&\ls \abs{a(u-u_I,v_h)} + \abs{k\pd{u-u_I,v_h}} + k\norm{P_hu-u_I}_0\cdot k\norm{v_h}_0. \nn}
Then the estimate \eqref{eq:nabla-phu-ui} follows from Lemma~\ref{uui}, \eqref{eqend} and the fact that
\eqn{ k\norm{u_h^+-u_I}_0\cdot k\norm{v_h}_0 &\leq \big(k\norm{u_h^+-u}_0+k\norm{u-u_I}_0\big)\norme{v_h}_{1,h}\\
&\ls (kh)^2\big(1+\sqrt{kh}\big)\norme{v_h}_{1,h} C_{u,g}. }
\end{proof}

Similar to \eqref{ghuh}, we have
\eq{ \norm{G_hP_hu-\na u}_0 \ls \norm{\na(P_hu-u_I)}_0 + \norm{G_hu_I-\na u}_0.  \label{ghphu}}
Then from Lemma~\ref{naphu}, Theorem~\ref{main0} and \eqref{ghphu}, we can obtain the following theorem.
\begin{theorem} \label{main2}
Assume that $\T_h$ satisfies the \emph{$\al$ approximation condition}.
Let $u$ and $P_hu$ be the solutions to \eqref{eq1.1a}--\eqref{eq1.1b} and \eqref{ellpro}, respectively.
Then the following error estimate holds:
\eq{ \norm{G_hP_hu-\na u}_0\ls \big(kh^{1+\al}+(kh)^2\big)  C_{u,g}. \label{phusup} }
\end{theorem}

From Lemma~\ref{ellemm}, we {see} that the elliptic projection of $u$ is not \emph{polluted}.
As a result, the second {term on the right-hand} side of the estimate \eqref{phusup} is $(kh)^2$
instead of $k(kh)^2$. {Furthermore}, the error estimate of the elliptic projection of $u$
does not require the mesh condition $k(kh)^2\leq C_0$.

\begin{theorem}\label{main3}
Assume that $\T_h$ satisfies the \emph{$\al$ approximation condition}.
Let $u_h$ be the linear finite element solution to the problem \eqref{eq1.1a}--\eqref{eq1.1b}. We have
\eq{ \norm{G_hu_h-\na u_h}_0\ls \big(kh+k(kh)^3\big) C_{u,g}. \label{eq:th2eq1} }
\end{theorem}
\begin{proof}
We write $u_h=P_hu+(u_h-P_hu):=P_hu+\ta_h$, where $P_hu$ is defined by \eqref{ellpro}.
Then by the triangle inequality we have
\eq{ \norm{G_hu_h-\na u_h}_0 &= \norm{G_h(P_hu+\ta_h)-\na (P_hu+\ta_h)}_0\label{main3eq1}\\
&\ls \norm{G_hP_hu-\na P_hu}_0 + \norm{G_h\ta_h-\na \ta_h}_0. \nn}
From Lemma~\ref{ellemm},
\eqn{ \norm{\na u-\na P_hu}_0^2 &\ls \inf_{v_h\in V_h} \norme{u-v_h}^2 \ls \abs{u-u_I}_1^2 + k^2\norm{u-u_I}_0^2 \\
&\ls h^2(1+k^2h^2)\norm{u}_2^2 \ls (kh)^2 (1+(kh)^2) C_{u,g}^2, }
which imples that from Theorem~\ref{main2},
\eq{ \norm{G_hP_hu-\na P_hu}_0 &\leq \norm{G_hP_hu-\na u}_0 + \norm{\na u-\na P_hu}_0 \label{main3eq2}\\
&\ls (kh^{1+\al}+(kh)^2)C_{u,g} + khC_{u,g} \nn\\
&\ls (kh+(kh)^2)C_{u,g}. \nn }
From \eqref{fem} and \eqref{ellpro}, we see that $\ta_h$ satisfies
\eq{ a(\ta_h,v_h) + \i k\pd{\ta_h,v_h} = - k^2(u-u_h,v_h). }
It is easy to see that $\ta_h$ can be understood as the finite element solution to the following Poisson problem
with Robin boundary:
\eqn{ -\De \ta &= -k^2(u-u_h) \quad\mathrm{in} \quad \Om,\\
\frac{\pa \ta}{\pa n} + \i k\ta &= 0\quad\quad\quad\quad\qquad\mathrm{on}\quad \Ga.  }
Therefore,
\eq{ \norm{G_h\ta_h-\na \ta_h}_0 &\leq \norm{G_h\ta_h-\na \ta}_0 +  \norm{\na\ta-\na \ta_h}_0\label{main3eq3}\\
&\ls h\norm{\ta}_2\ls k^2h\norm{u-u_h}_0\nn\\
&\ls k^2h(kh^2+k^2h^2)C_{u,g}\nn\\
&\ls ((kh)^3+k(kh)^3)C_{u,g}.\nn }
The proof is completed by combining \eqref{main3eq1}--\eqref{main3eq3}.
\end{proof}

\emph{Remark} 5.1. We emphasize that the estimate of $\norm{G_hu_h-\na u_h}_0$ may not be perfect, although
$k(kh)^3$ is less than the pollution error $k(kh)^2$ of the operator $G_h$. What we expect is the estimate
$\norm{G_hu_h-\na u_h}_0\ls khC_{u,g}$, which coincides with our numerical tests in the next section.
However, if invoking the mesh condition $k(kh)^2\leq C_0$, we have $\norm{G_hu_h-\na u_h}_0\ls kh(1+C_0)C_{u,g}$, which indicates that the pollution error between these two quantities are ``almost" cancelled.

\section{Numerical examples} \label{num}
In this section, we will verify our theoretical results by simulating the following two-dimensional Helmholtz problem:
\eq{ -\De u - k^2 u &=f:=\frac{\sin(kr)}{r}  \qquad\mbox{in  } \Om, \label{num-eq-1}\\
\frac{\pa u}{\pa n} +\i k u& =g \quad\qquad\qquad\qquad\mbox{on } \Ga. \label{num-eq-2} }
$g$ is so chosen that the exact solution
is
\eq{ u=\frac{\cos(k r)}{r}-\frac{\cos k+\i\sin k}{k\big(J_0(k)+\i J_1(k)\big)}J_0(k r) }
in polar coordinates, where $J_\nu(z)$ are Bessel functions of the first kind.
\begin{figure}[htbp]
\begin{center}
\includegraphics[width=0.5\textwidth]{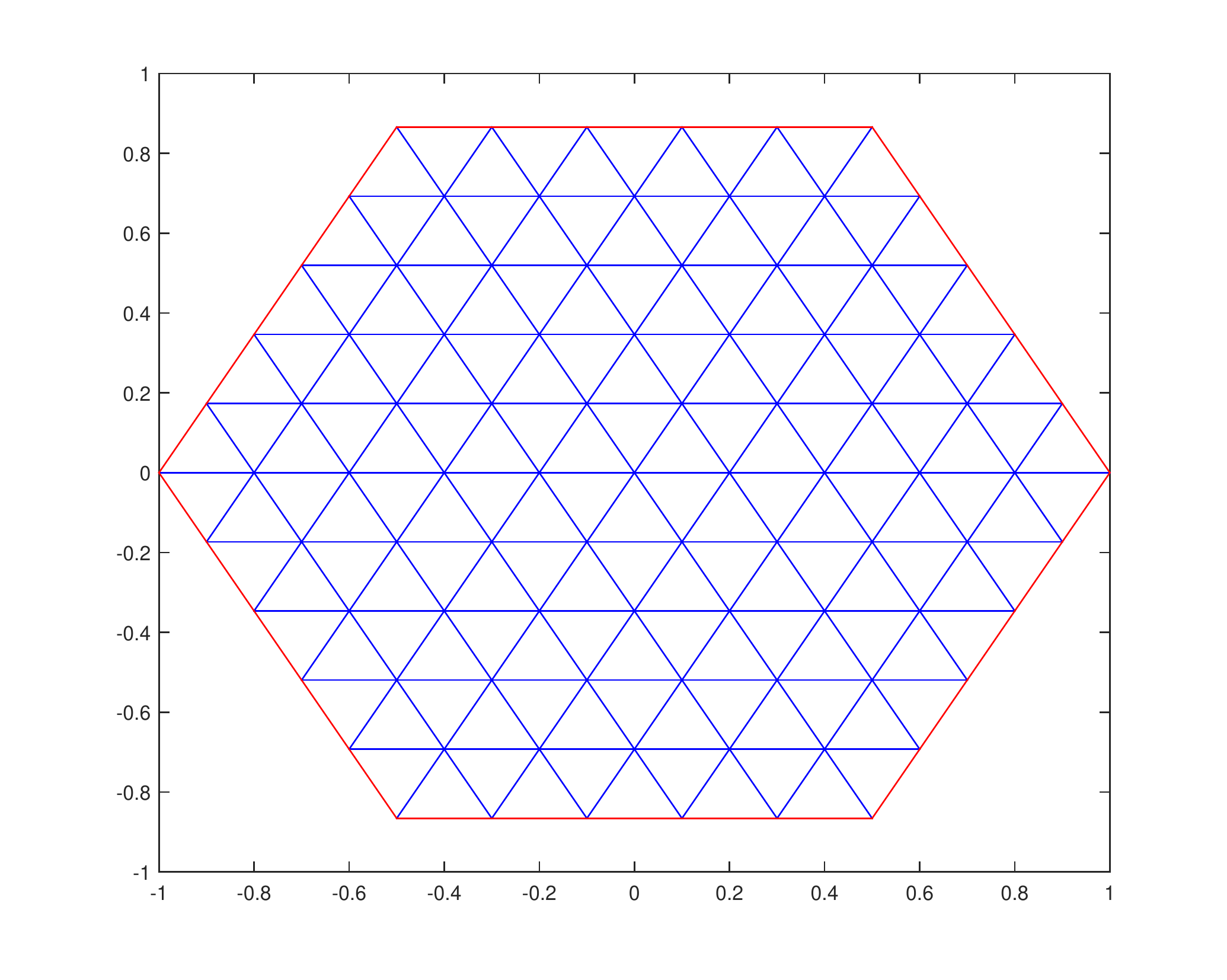}
\caption{Geometry and a sample mesh $\T_{1/5}$ that consists of congruent and equilateral triangles of size
$h=1/5$ for the example.}
\end{center}
\label{num-fig1}
\end{figure}

This problem has been computed in \cite{fw09,dw,zd,dzh} by the finite element method, the continuous interior penalty finite element
method and the interior penalty discontinuous Galerkin method on both triangular meshes and rectangular meshes.

\subsection{Errors of $\na u_h$ and $G_h u_h$ }
Let $\Om$ be the unit regular hexagon with center (0,0) (cf. Figure~\ref{num-fig1}).
For any positive integer $m$, let $\T_{1/m}$ be the regular triangulation that consists of $6m^2$ congruent and equilateral
triangles of size $h=1/m$. See Figure~\ref{num-fig1} for a sample triangulation $\T_{1/5}$. Let $u_h$ be the linear finite element solution in this subsection.

From Theorem~\ref{main1}, the error of the recovered {gradient in} the $H^1$-seminorm is bounded by
\eq{ \norm{G_hu_h-\na u}_0\leq C_1kh^{1+\al} + C_2k(kh)^2 \label{num-eq-3} }
for some constants $C_1$ and $C_2$ if $k(kh)^2\leq C_0$. The second {term} on the right-hand side of \eqref{num-eq-3} is the
so-called pollution error. We actually have
$\al=1$ because of the congruent and equilateral triangles in the meshes (see Figure~\ref{num-fig1}).
We now verify the error {bound by numerical} results.

\begin{figure}[htbp]
\begin{center}
\includegraphics[width=0.49\textwidth]{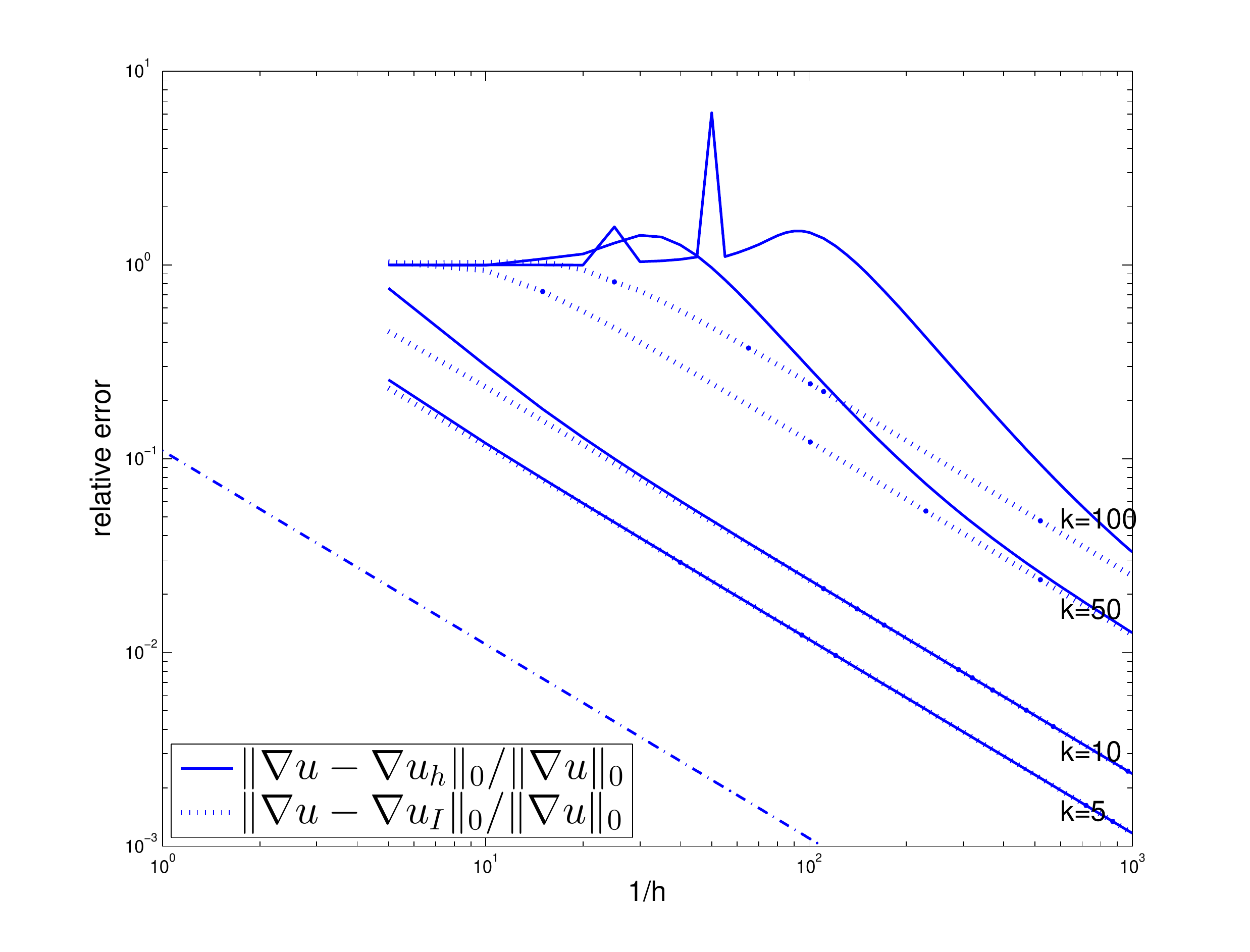}
\includegraphics[width=0.49\textwidth]{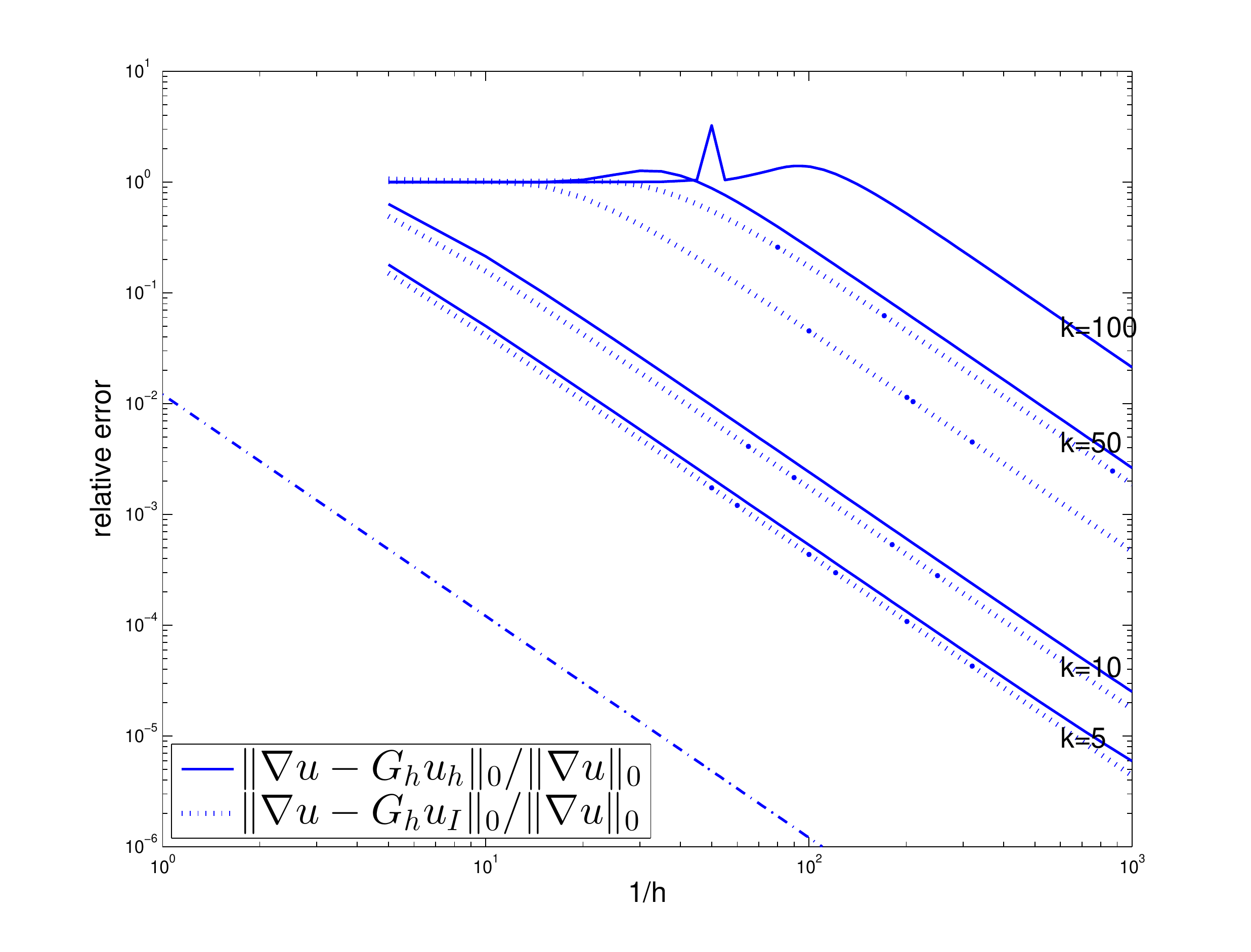}
\caption{Left graph: the relative error of the finite element solution (solid) and the relative error of the finite element
interpolant (dotted) in $H^1$-seminorm for $k=5,k=10,k=50$, and $k=100$, respectively. Right graph: the relative error of the recovered gradient of
the finite element solution and that of the finite element interpolant for $k=5,k=10,k=50$, and $k=100$, respectively.
Dash--dot lines give reference slopes $-1$ and $-2$, respectively.  }
\label{num-fig2}
\end{center}
\end{figure}

\begin{figure}[htbp]
\begin{center}
\includegraphics[width=0.5\textwidth]{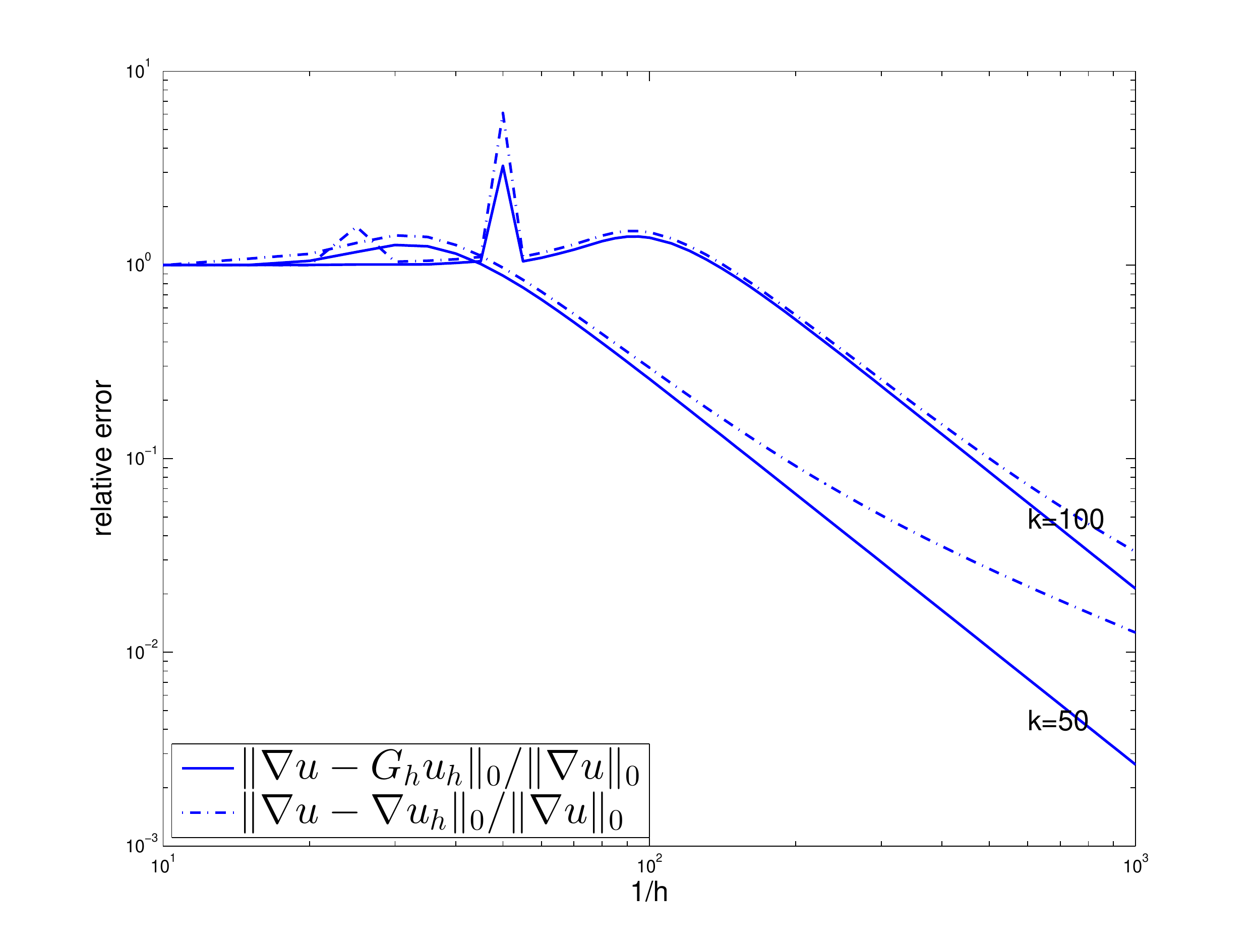}
\caption{The relative error of the finite element solution (dash--dot) and that of the recovered gradient of the finite element solution (solid).}\label{num-fig2a}
\end{center}
\end{figure}

In the left graph of Figure~\ref{num-fig2}, the $H^1$-seminorm relative error of the finite element solution
and that of the linear interpolant are displayed simultaneously.
When $k=5,10$, the relative error of the finite element solution is similar to that of
the interpolant. When $k=50,100$, we can easily detect
the existence of the pollution error and the effect of the mesh condition $k(kh)^2\leq C_0$.
By contrast, as shown in the right graph of Figure~\ref{num-fig2}, the gradient recovery method converges faster than the linear interpolant and the relative error of the recovered gradient decays at rate $h^2$ (slope $-2$ in the log-log scale)
 for $k=5, 10$. For $k=50, 100$, the relative error of the recovered gradient
stays around $100\%$ (no-convergence) and then decays at rate $h^2$. We notice that the ``no-convergence range" increases with $k$. In addition, Figure~\ref{num-fig2a} shows that the decaying points of both the gradient of the finite element solution $\na u_h$ and the recovered gradient $G_hu_h$ are the same, which indicates that their convergence mesh conditions are the same, that is $k(kh)^2\leq C_0$, and
the pollution effect is still there for the recovered gradient. We remark that the numerical tests for the pollution phenomenon of the finite element method have been done largely in the literature.
For more details, a reader is refer to \cite{dw} and references therein.

Next we verify more precisely the pollution term in \eqref{num-eq-3}. To do so, we introduce the definition of the critical mesh size with respect to
a given relative tolerance \cite{w,dw}.
\begin{definition}
Given a relative tolerance $\ep$, a wave number $k$, the critical mesh size $h(k,\ep)$ with respect to the relative tolerance $\ep$ is defined by the
maximum mesh size such that the relative error of the finite element solution in {the} $H^1$-seminorm (or the relative error of recovered gradient of the finite element solution in {the} $H^0$-norm) is less than or equal to $\ep$
\end{definition}

Clearly, if the pollution terms in \eqref{fempoll} and \eqref{num-eq-3} are of order $k^3h^2$, then $h(k,\ep)$
should be proportional to $k^{-3/2}$ for $k$ large enough. This is verified by Figure~\ref{num-fig3}.
So our theoretical result is sharp with respect to $k$ and $h$.

To compare the pollution errors more intuitively, we plot
the relative error of the finite element solution in the $H^1$-seminorm
and the relative error of the recovered gradient in the $H^0$-norm for
$k=1,2,\cdots,600$ with fixed $kh=1$ and $kh=1/2$ in the left graph of Figure~\ref{num-fig4}.
We see that both relative errors increase linearly in the pre-asymptotical range $k(kh)^2\leq C_0$:
for $kh=1/2$, the range is about $k\leq500$, and for $kh=1$, the range is much less with $k\leq100$.
However, we do not know the behaviors of these relative errors when $k$ is much larger theoretically.

To investigate further the influence of PPR to the pollution errors, we estimate the error between the gradient of the finite element solution and its recovered gradient
and prove that this error is controlled by $(kh+k(kh)^3)C_{u,g}$ (cf. Theorem~\ref{main3}).
In the right graph of Figure~\ref{num-fig4}, we depict this error for $kh=1$ and $kh=1/2$, respectively. We see that both of them are dominated by $kh$, which indicates that the pollution error is ``almost" cancelled
between the gradient of the finite element solution and its recovered gradient.
So our estimate in Theorem~\ref{main3} is pre-asymptotically correct under the mesh condition $k(kh)^2\leq C_0$.
However, the relative error estimate under other mesh condition is still unknown and deserves further study in the future.

Both graphs of Figure~\ref{num-fig4} imply that the gradient recovery method does not reduce the pollution error of
the finite element solution. Nevertheless, it does
improve the low order term in \eqref{fempoll}.

\begin{figure}[htbp]
\begin{center}
\includegraphics[width=0.49\textwidth]{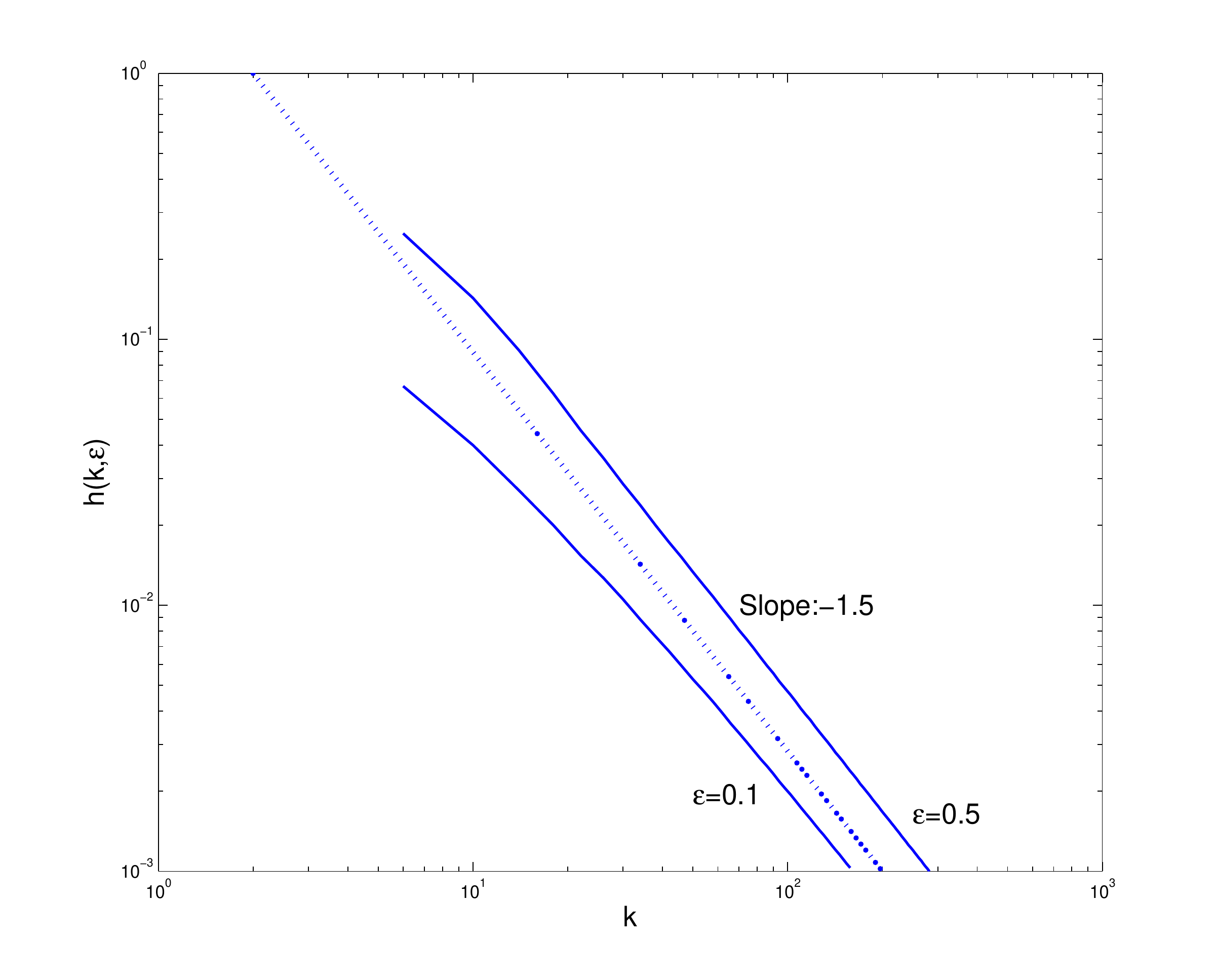}
\includegraphics[width=0.49\textwidth]{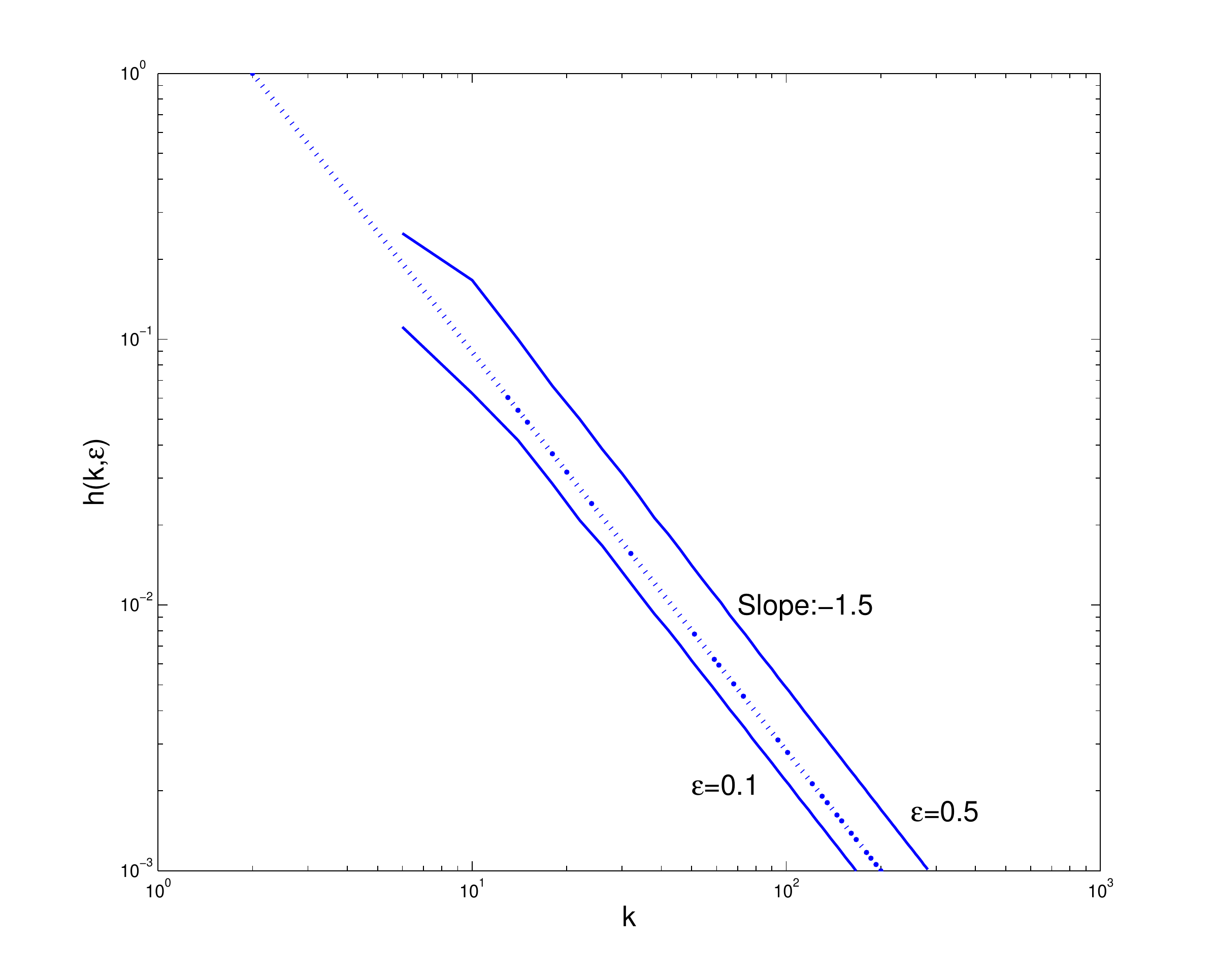}
\caption{$h(k,0.5)$ and $h(k,0.1)$ versus $k$ for the finite element solution (left) and for the recovered gradient of the finite element solution (right), respectively. The dotted lines give lines of slope $-1.5$ in the log-log scale.}
\label{num-fig3}
\end{center}
\end{figure}

\begin{figure}[htbp]
\begin{center}
\includegraphics[width=0.49\textwidth]{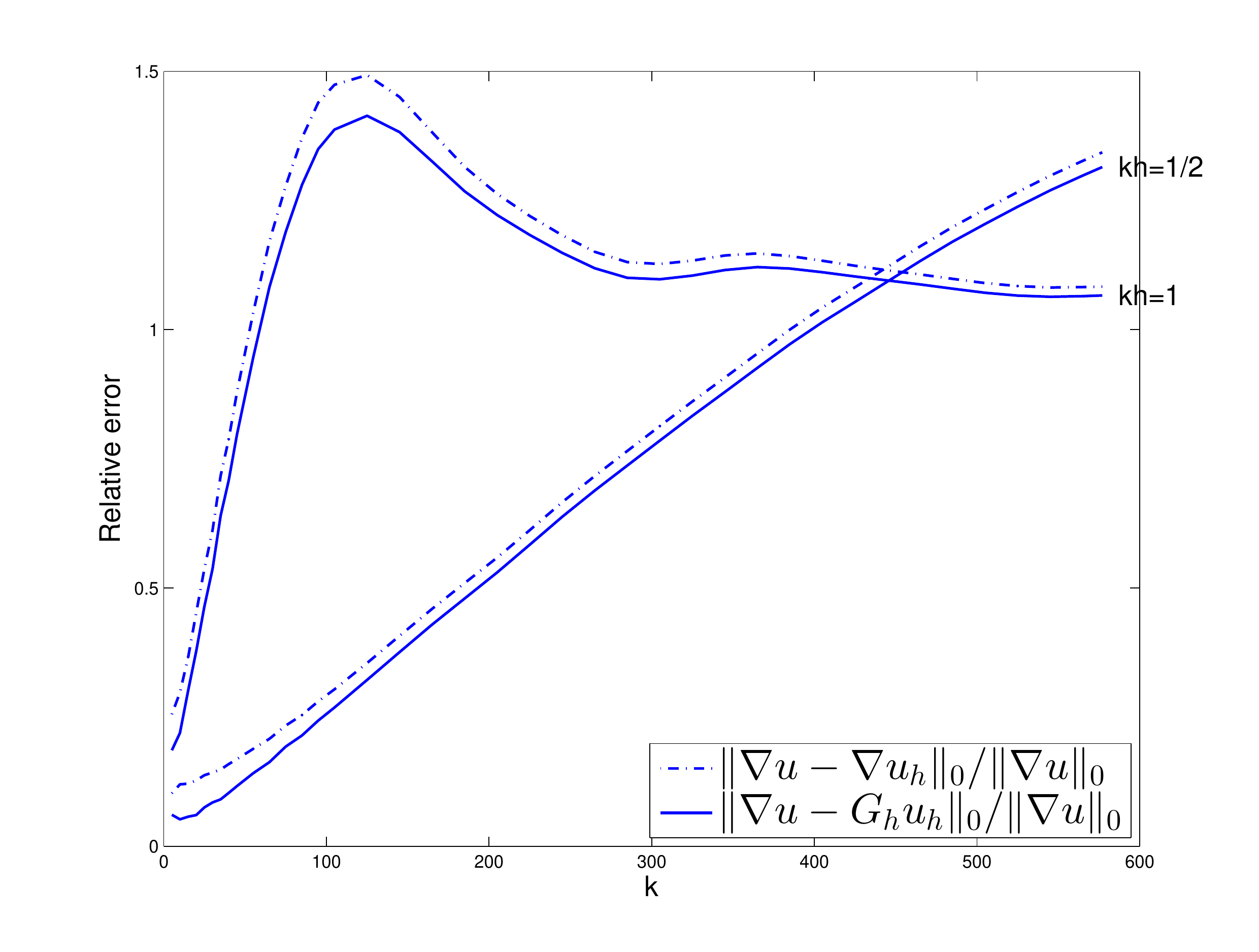}
\includegraphics[width=0.49\textwidth]{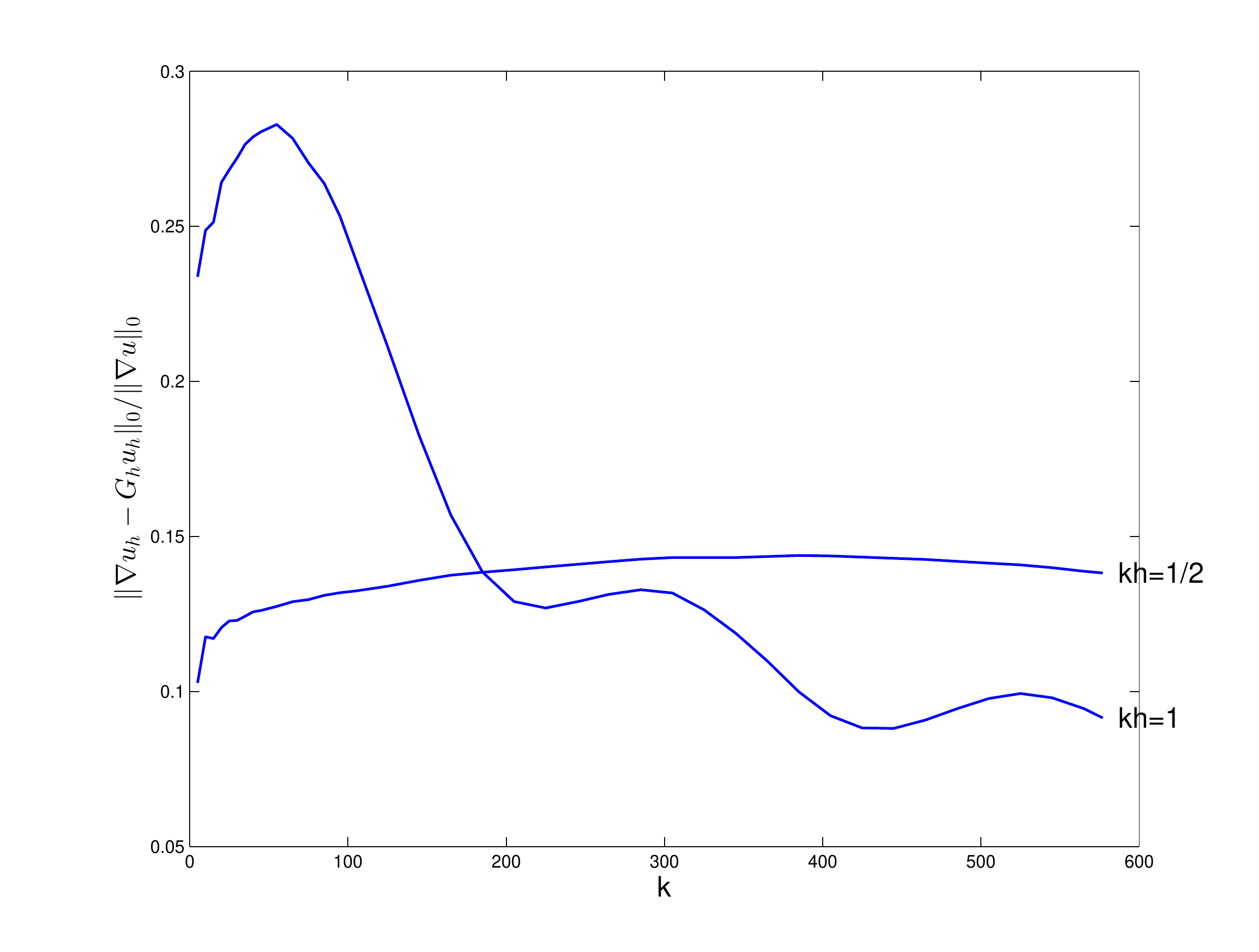}
\caption{Left graph: the relative error of the finite element solution in $H^1$-seminorm (dash-dot lines) and that of the recovered gradient of the finite element solution in $H^0$-norm (solid lines) with mesh size $h$ determined by $kh=1$ and $kh=1/2$. Right graph: the estimate between $G_hu_h$ and $\na u_h$ with mesh size $h$ determined by $kh=1$ and $kh=1/2$.}
\label{num-fig4}
\end{center}
\end{figure}

\subsection{Richardson extrapolation and the a posteriori error estimator} \label{sec_richardson_posteriori}
The Richardson extrapolation method is an efficient procedure to raise the accuracy of numerical methods, such as finite difference \cite{ms83} and finite element methods \cite{br86,w89,h83,lsy98}. In this subsection, we apply the Richardson extrapolation to the gradient and recovered gradient of the finite element solution, respectively.
Our motivation is based on the error estimate (\ref{ghpoll}) in Theorem \ref{main1}.
With $\alpha=1$, both the interpolation error and pollution error are of order $h^2$ (with different powers of $k$),
and hence Richardson extrapolation would work on both terms simultaneously.
Since the powers of $h$ in the two terms of (\ref{fempoll}) are not balanced, the Richardson extrapolation will not work without using PPR recovery.
Our numerical tests clearly demonstrate this difference as we shall see in the subsequence.

Let $\T_h$ be a uniform and regular triangulation and let $\T_{h/2}$ be generated from $\T_h$ by dividing each trangle as usual into four congruent subtriangles.

Define the Richardson extrapolation operator $R$ by
\eqn{Rv_{h/2}|_K=(4v_{h/2}-v_{h})/3\quad\forall K\in\T_{h/2},}
where $v_h$ is a piecewise polynomial function over $\T_h$.

We first simulate the probelm~\eqref{num-eq-1}--\eqref{num-eq-2} over regular pattern uniform triangulations $\T_h$ of the unit square $[0,1]\times[0,1]$.

\begin{table}
\begin{center}
\begin{tabular}{|c|c|c|c|c|c|}
\hline\noalign{\smallskip}
h & $\na u_h$ & $R\na u_h$ & $G_hu_h$ & $RG_hu_h$ & $G_hu_I$  \\
\noalign{\smallskip}\hline\noalign{\smallskip}
1/4 &      9.5808e-01   &              &  9.5511e-01 &             &     \\
1/8 &      5.8248e-01   &  6.0502e-01  &  5.3014e-01 &  4.5896e-01 & 3.9074e-01\\
1/16 &     2.6521e-01   &  2.6763e-01  &  1.8339e-01 &  9.4611e-02 & 1.1620e-01\\
1/32 &     1.2121e-01   &  1.3214e-01  &  5.0599e-02 &  1.2457e-02 & 2.9841e-02\\
1/64 &     5.8610e-02   &  6.6580e-02  &  1.2986e-02 &  2.1927e-03 & 7.4567e-03\\
1/128 &    2.9033e-02   &  3.3383e-02  &  3.2693e-03 &  5.0283e-04 & 1.8578e-03\\
1/256 &    1.4482e-02   &  1.6704e-02  &  8.1935e-04 &  1.2149e-04 & 4.6332e-04\\
1/512 &    7.2365e-03   &  8.3538e-03  &  2.0524e-04 &  2.9832e-05 & 1.1566e-04\\
1/1024 &   3.6177e-03   &  4.1771e-03  &  5.1531e-05 &  7.4156e-06 & 2.8894e-05\\
\noalign{\smallskip}\hline
\end{tabular}
\caption{$\norm{\na u-\na u_h}_0/\abs{u}_1$, $\norm{\na u-R\na u_h}_0/\abs{u}_1$, $\norm{\na u-G_h u_h}_0/\abs{u}_1$, $\norm{\na u-RG_H u_h}_0/\abs{u}_1$ and $\norm{\na u-G_h u_I}_0/\abs{u}_1$ over $\T_h$ for $k=10$.}
\label{tab1}
\end{center}
\end{table}

\begin{table}
\begin{center}
\begin{tabular}{|c|c|c|c|c|c|}
\hline\noalign{\smallskip}
h & $\na u_h$ & $R\na u_h$ & $G_hu_h$ & $RG_hu_h$ & $G_hu_I$  \\
\noalign{\smallskip}\hline\noalign{\smallskip}
1/4 &      9.9910e-01   &              &  9.9873e-01 &             &           \\
1/8 &      1.0048e+00   &  1.0087e+00  &  1.0056e+00 &  1.0096e+00 & 1.0055e+00\\
1/16 &     1.0693e+00   &  1.1290e+00  &  9.9738e-01 &  1.0036e+00 & 1.0014e+00\\
1/32 &     1.1929e+00   &  1.3541e+00  &  1.0570e+00 &  1.1145e+00 & 5.9028e-01\\
1/64 &     1.1021e+00   &  1.3134e+00  &  1.0128e+00 &  1.1568e+00 & 1.8951e-01\\
1/128 &    3.9158e-01   &  3.9411e-01  &  3.5433e-01 &  3.2505e-01 & 5.0046e-02\\
1/256 &    1.2126e-01   &  9.5288e-02  &  9.2998e-02 &  2.9092e-02 & 1.2631e-02\\
1/512 &    4.5197e-02   &  4.4683e-02  &  2.3462e-02 &  2.0761e-03 & 3.1591e-03\\
1/1024 &   2.0172e-02   &  2.2284e-02  &  5.8762e-03 &  2.2653e-04 & 7.8911e-04\\
\noalign{\smallskip}\hline
\end{tabular}
\caption{$\norm{\na u-\na u_h}_0/\abs{u}_1$, $\norm{\na u-R\na u_h}_0/\abs{u}_1$, $\norm{\na u-G_h u_h}_0/\abs{u}_1$, $\norm{\na u-RG_H u_h}_0/\abs{u}_1$ and $\norm{\na u-G_h u_I}_0/\abs{u}_1$ over $\T_h$ for $k=50$.}
\label{tab2}
\end{center}
\end{table}

Table~\ref{tab1} shows the relative $L^2$-norm errors of $\na u_h$, $R\na u_h$ and their Richardson extrapolation in the case $k=10$.
As we expected,  $\norm{\na u-\na u_h}_0/\abs{u}_1$ converges at rate $O(h)$ and $\norm{\na u-G_hu_h}_0/\abs{u}_1$ decays at rate $O(h^2)$.
We can observe that the relative error of $R\na u_h$ is worse than that of $\na u_h$ and the relative error of $RG_hu_h$ is much better than that of $G_hu_h$.
For a larger wave number $k=50$, the relative errors are shown in Table~\ref{tab2}. The data demonstrate similar behaviors of numerical solutions to those in Table~\ref{tab1}
when the mesh size is sufficient small.

The good behavior of the operator $RG_h\cdot$ in Table~\ref{tab1} and Table~\ref{tab2} makes it possible to define the following {\it a posteriori} error estimator
\eq{\eta_h = \norm{RG_hu_h-\na u_h}_0.\label{posteriori}}
Table~\ref{tab3} and Table~\ref{tab4} illustrate the asymptotic exactness of the error estimator based on the recovery operator $G_h$ and the extrapolation operator $R$.
\begin{table}
\begin{center}
\begin{tabular}{|c|cc|cc|}
\hline\noalign{\smallskip}
&\multicolumn{2}{|c|}{k=10}& \multicolumn{2}{|c|}{k=30}\\
\hline\noalign{\smallskip}
h & $\norm{\na u-\na u_h}_0$ & $\eta_h$ & $\norm{\na u-\na u_h}_0$ & $\eta_h$  \\
\noalign{\smallskip}\hline\noalign{\smallskip}
1/4 &      7.9165e-01   &              &  8.9499e-01 & \\
1/8 &      4.8127e-01   &  3.5128e-01  &  8.8726e-01 & 2.8705e-01\\
1/16&     2.1913e-01   &  1.9677e-01  &  9.9983e-01 & 4.2429e-01\\
1/32 &     1.0015e-01   &  9.8802e-02  &  7.9908e-01 & 3.1293e-01\\
1/64 &     4.8426e-02   &  4.8414e-02  &  2.9350e-01 & 2.2032e-01\\
1/128 &    2.3988e-02   &  2.3994e-02  &  1.0199e-01 & 9.7829e-02\\
1/256 &    1.1965e-02   &  1.1966e-02  &  4.2406e-02 & 4.2259e-02\\
1/512 &    5.9791e-03   &  5.9793e-03  &  1.9948e-02 & 1.9945e-02\\
1/1024 &   2.9891e-03   &  2.9891e-03  &  9.8094e-03 & 9.8098e-03\\
\noalign{\smallskip}\hline
\end{tabular}
\caption{The errors of $\na u_h$ and the a posteriori error estimator over $\T_h$ for $k=10$ and $k=30$}\label{tab3}
\end{center}
\end{table}

\begin{table}
\begin{center}
\begin{tabular}{|c|cc|cc|}
\hline\noalign{\smallskip}
&\multicolumn{2}{|c|}{k=60}& \multicolumn{2}{|c|}{k=120}\\
\hline\noalign{\smallskip}
h & $\norm{\na u-\na u_h}_0$ & $\eta_h$ & $\norm{\na u-\na u_h}_0$ & $\eta_h$  \\
\noalign{\smallskip}\hline\noalign{\smallskip}
1/4 &    8.3466e-01 &            & 8.2188e-01 &           \\
1/8 &    8.8755e-01 & 6.1990e-02 & 8.5109e-01 & 1.5580e-02\\
1/16 &   9.0952e-01 & 2.8306e-01 & 8.7877e-01 & 4.7890e-02\\
1/32 &   9.9385e-01 & 3.7793e-01 & 9.2054e-01 & 2.9554e-01\\
1/64 &   1.1260e+00 & 2.9684e-01 & 9.8618e-01 & 3.4281e-01\\
1/128 &  5.4450e-01 & 3.1504e-01 & 1.0975e+00 & 2.6629e-01\\
1/256 &  1.6186e-01 & 1.4898e-01 & 9.6799e-01 & 3.3001e-01\\
1/512 &  5.3586e-02 & 5.3062e-02 & 3.0027e-01 & 2.5931e-01\\
1/1024 & 2.1947e-02 & 2.1932e-02 & 8.3593e-02 & 8.2496e-02\\
\noalign{\smallskip}\hline
\end{tabular}
\caption{The errors of $\na u_h$ and the a posteriori error estimates over $\T_h$ for $k=60$ and $k=120$}
\label{tab4}
\end{center}
\end{table}

Next we turn to the Delaunay triangulation over the unit square $\Om$ and L-shaped domain $\Om_L=\Om\backslash[0.5,1]\times[0.5,1]$.
The initial mesh $\T^D_0$ is obtained by using a Delaunay triangulation algorithm. Then $\T^D_j$ is obtained from $\T^D_{j-1}$ by dividing each
triangle into four congruent triangles. Data in Tables \ref{tab1_Delaunay}--\ref{tab4_Delaunay} show the superconvergence of the recovered gradient at the rate of $O(h^2)$
(see the fourth columns) and the asymptotic exactness of $\eta_h$ (see the sixth columns) over Delaunay triangulations. Therefore, the PPR method combined with the Richardson extrapolation performs very well and leads to an {\it a posteriori} error estimator.
\begin{table}
\begin{center}
\begin{tabular}{|c|c|c|c|c|c|}
\hline\noalign{\smallskip}
m& DOF & $\norm{\na u-\na u_h}_0$ & $\norm{\na u-G_hu_h}_0$ & $\norm{\na u-RG_hu_h}_0$ & $\eta_h$\\
\noalign{\smallskip}\hline\noalign{\smallskip}
0 & 54      & 4.1028e-01 &  4.2562e-01 &             &            \\
1 & 193     & 1.8409e-01 &  1.4508e-01 &  8.2579e-02 &  1.8430e-01\\
2 & 729     & 8.7025e-02 &  3.8073e-02 &  1.2520e-02 &  8.7486e-02\\
3 & 2833    & 4.2749e-02 &  9.4214e-03 &  2.6138e-03 &  4.2840e-02\\
4 & 11169   & 2.1275e-02 &  2.3269e-03 &  5.7444e-04 &  2.1286e-02\\
5 & 44353   & 1.0625e-02 &  5.7802e-04 &  1.3013e-04 &  1.0626e-02\\
6 & 176769  & 5.3111e-03 &  1.4423e-04 &  3.0557e-05 &  5.3112e-03\\
7 & 705793  & 2.6553e-03 &  3.6082e-05 &  7.3719e-06 &  2.6554e-03\\
\noalign{\smallskip}\hline
\end{tabular}
\caption{The errors of $\na u_h$, $G_hu_h$, $RG_hu_h$ and the a posteriori error estimates over the Delaunay triangulation $\T^D_m$ of the unit square $\Om$ ($m=0,1,2,\ldots,7$) for $k=10$.}
\label{tab1_Delaunay}
\end{center}
\end{table}

\begin{table}
\begin{center}
\begin{tabular}{|c|c|c|c|c|c|}
\hline\noalign{\smallskip}
m& DOF & $\norm{\na u-\na u_h}_0$ & $\norm{\na u-G_hu_h}_0$ & $\norm{\na u-RG_hu_h}_0$ & $\eta_h$\\
\noalign{\smallskip}\hline\noalign{\smallskip}
0 & 54    &  8.8926e-01 &  8.8934e-01 &             &            \\
1 & 193   &  9.4715e-01 &  8.7929e-01 &  8.8957e-01 &  4.0315e-01\\
2 & 729   &  9.4889e-01 &  8.7838e-01 &  8.9673e-01 &  2.9724e-01\\
3 & 2833  &  1.0437e+00 &  9.5886e-01 &  1.0742e+00 &  2.9106e-01\\
4 & 11169 &  3.7904e-01 &  3.4890e-01 &  3.1643e-01 &  2.6232e-01\\
5 & 44353 &  1.1633e-01 &  9.2797e-02 &  2.9542e-02 &  1.0972e-01\\
6 & 176769&  4.2413e-02 &  2.3489e-02 &  2.3164e-03 &  4.2156e-02\\
7 & 705793&  1.8666e-02 &  5.8886e-03 &  3.1415e-04 &  1.8661e-02\\
\noalign{\smallskip}\hline
\end{tabular}
\caption{The errors of $\na u_h$, $G_hu_h$, $RG_hu_h$ and the a posteriori error estimates over the Delaunay triangulation $\T^D_m$ of the unit square $\Om$ ($m=0,1,2,\ldots,7$) for $k=60$.}
\label{tab2_Delaunay}
\end{center}
\end{table}

\begin{table}
\begin{center}
\begin{tabular}{|c|c|c|c|c|c|}
\hline\noalign{\smallskip}
m& DOF & $\norm{\na u-\na u_h}_0$ & $\norm{\na u-G_hu_h}_0$ & $\norm{\na u-RG_hu_h}_0$ & $\eta_h$\\
\noalign{\smallskip}\hline\noalign{\smallskip}
   0  & 279   &   1.1742e-01 &  7.5574e-02 &             &            \\
   1  & 1057  &   5.8161e-02 &  1.8932e-02 &  7.9039e-03 &  5.8495e-02\\
   2  & 4113  &   2.9051e-02 &  4.7206e-03 &  1.6396e-03 &  2.9076e-02\\
   3  & 16225 &   1.4529e-02 &  1.1963e-03 &  3.7617e-04 &  1.4530e-02\\
   4  & 64449 &   7.2656e-03 &  3.0479e-04 &  8.8224e-05 &  7.2656e-03\\
   5  & 256900 &   3.6331e-03 &  7.7808e-05 &  2.1327e-05 &  3.6331e-03\\
   6  & 1025800 &   1.8166e-03 &  1.9875e-05 &  5.2388e-06 &  1.8166e-03\\
\noalign{\smallskip}\hline
\end{tabular}
\caption{The errors of $\na u_h$, $G_hu_h$, $RG_hu_h$ and the a posteriori error estimates over the Delaunay triangulation $\T^D_m$ of the L-shaped domian $\Om_L$ ($m=0,1,2,\ldots,7$) for $k=10$.}
\label{tab3_Delaunay}
\end{center}
\end{table}

\begin{table}
\begin{center}
\begin{tabular}{|c|c|c|c|c|c|}
\hline\noalign{\smallskip}
m& DOF & $\norm{\na u-\na u_h}_0$ & $\norm{\na u-G_hu_h}_0$ & $\norm{\na u-RG_hu_h}_0$ & $\eta_h$\\
\noalign{\smallskip}\hline\noalign{\smallskip}
   0 &  279  & 9.3370e-01  &   8.0586e-01 &             &            \\
   1 &  1057 &  9.9962e-01 &   8.9924e-01 &  9.6330e-01 &  3.0810e-01\\
   2 &  4113 &  5.7234e-01 &   5.3249e-01 &  5.8243e-01 &  2.6218e-01\\
   3 &  16225&   1.8276e-01&   1.5775e-01 &  8.6317e-02 &  1.5438e-01\\
   4 &  64449&   6.2328e-02&   4.0796e-02 &  7.0117e-03 &  6.0954e-02\\
   5 &  256900&  2.5937e-02&   1.0278e-02 &  7.4238e-04 &  2.5894e-02\\
   6 &  1025800& 1.2217e-02&   2.5744e-03 &  1.4689e-04 &  1.2217e-02\\
\noalign{\smallskip}\hline
\end{tabular}
\caption{The errors of $\na u_h$, $G_hu_h$, $RG_hu_h$ and the a posteriori error estimates over the Delaunay triangulation $\T^D_m$ of the L-shaped domian $\Om_L$ ($m=0,1,2,\ldots,7$) for $k=60$.}
\label{tab4_Delaunay}
\end{center}
\end{table}

Finally, we use the a posteriori error estimator \eqref{posteriori} to simulate the Helmholtz problem
\eq{ -\De u - k^2 u &=\frac{\sin(k\tilde r)}{\tilde r}e^{-50\tilde r}  \qquad\mbox{in  }\quad \Om,\label{num-eq2a}\\
\frac{\pa u}{\pa n} +\i k u& =0 \quad\qquad\qquad\qquad\mbox{on }\quad \Ga,\label{num-eq2b}}
where $\Om$ is the unit square $[0,1]\times[0,1]$ and $\tilde r=\sqrt{(x-0.5)^2+(y-0.5)^2}$. Let $u_h$ be the linear finite element solution to the problem~\eqref{num-eq2a}--\eqref{num-eq2b} over $\T_m$.

We do not have the expression of the exact solution to the problem~\eqref{num-eq2a}--\eqref{num-eq2b}.  However, Table~\ref{tab5}
shows that the solutions are relatively accurate when the mesh sizes are greater than $128,512,1024$ for the wave numbers $k=30,60,120$, respectively.
The graphs of numerical solutions for different $k$ and $h$ in Figure~\ref{num-fig5}--\ref{num-fig7} illustrate the findings.

\begin{table}
\begin{center}
\begin{tabular}{|c|c|c|c|}
\hline\noalign{\smallskip}
m  & $k=30$ & $k=60$ & $k=120$  \\
\noalign{\smallskip}\hline\noalign{\smallskip}
8 &    5.6287e-03 & 1.4636e-04 & 6.2692e-05\\
16 &    2.0711e-02 & 2.1027e-02 & 3.6726e-04\\
32 &   2.2812e-02 & 4.0705e-02 & 3.6307e-02\\
64 &   1.4604e-02 & 3.6427e-02 & 5.1698e-02\\
128 &  7.1816e-03 & 2.6521e-02 & 4.9015e-02\\
256 &  3.4422e-03 & 1.2030e-02 & 4.3212e-02\\
512 &  1.6928e-03 & 5.1373e-03 & 2.0382e-02\\
1024 & 8.4244e-04 & 2.4126e-03 & 7.4914e-03\\
\noalign{\smallskip}\hline
\end{tabular}
\caption{ The a posteriori error estimator $\eta_h$ over $\T_m$ ($m=8,16,\ldots,1024$) for $k=30,60,120$}
\label{tab5}
\end{center}
\end{table}

\begin{figure}[htbp]
\begin{center}
\includegraphics[width=0.9\textwidth]{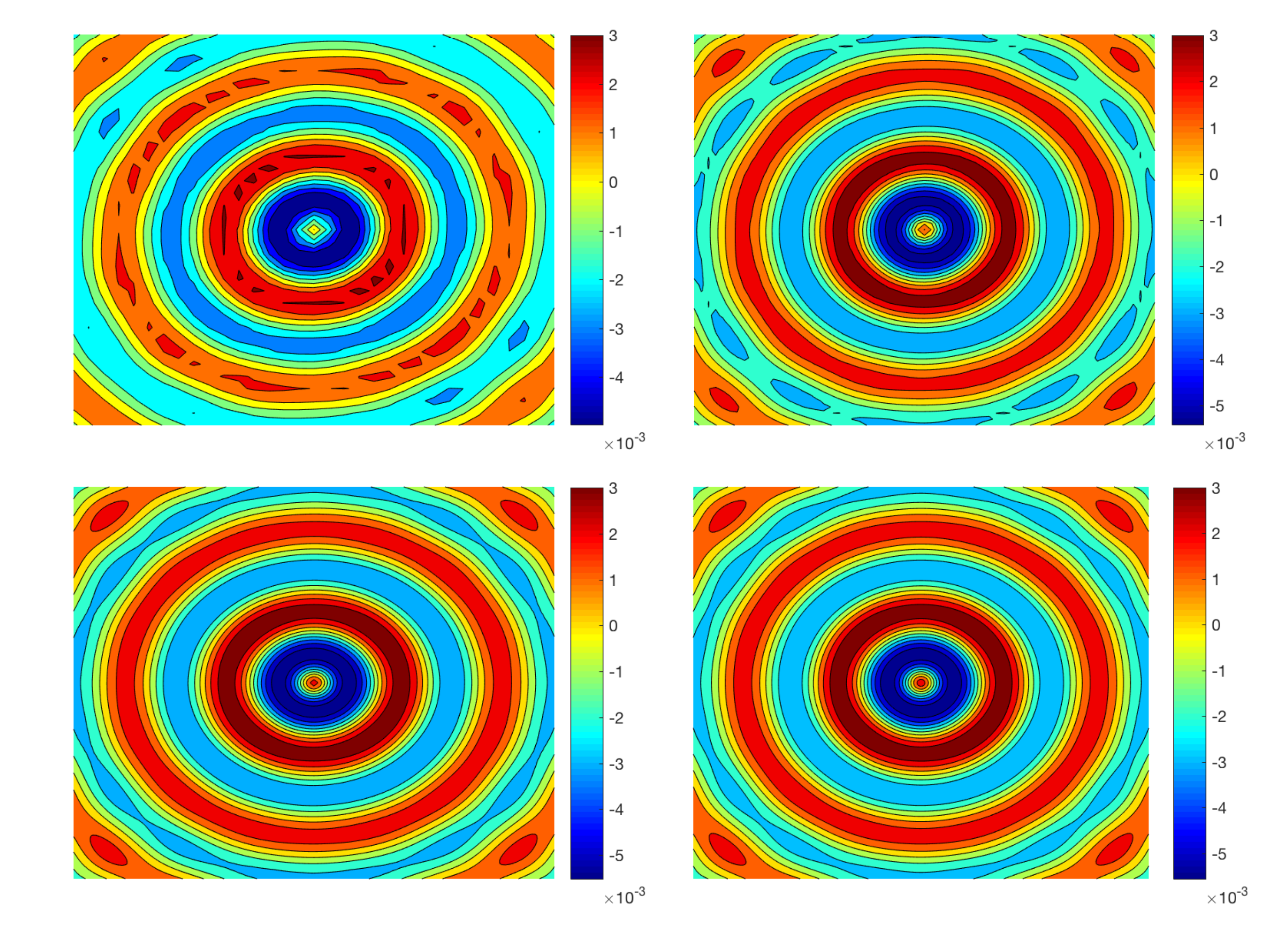}
\caption{The real part of the finite element solution for the equations~\eqref{num-eq2a}--\eqref{num-eq2b} for $k=30$ over $\T_m$ with $m=8$ (top left), $m=32$ (top right), $m=128$ (bottom left) and $m=1024$ (bottom right).}
\label{num-fig5}
\end{center}
\end{figure}

\begin{figure}[htbp]
\begin{center}
\includegraphics[width=0.9\textwidth]{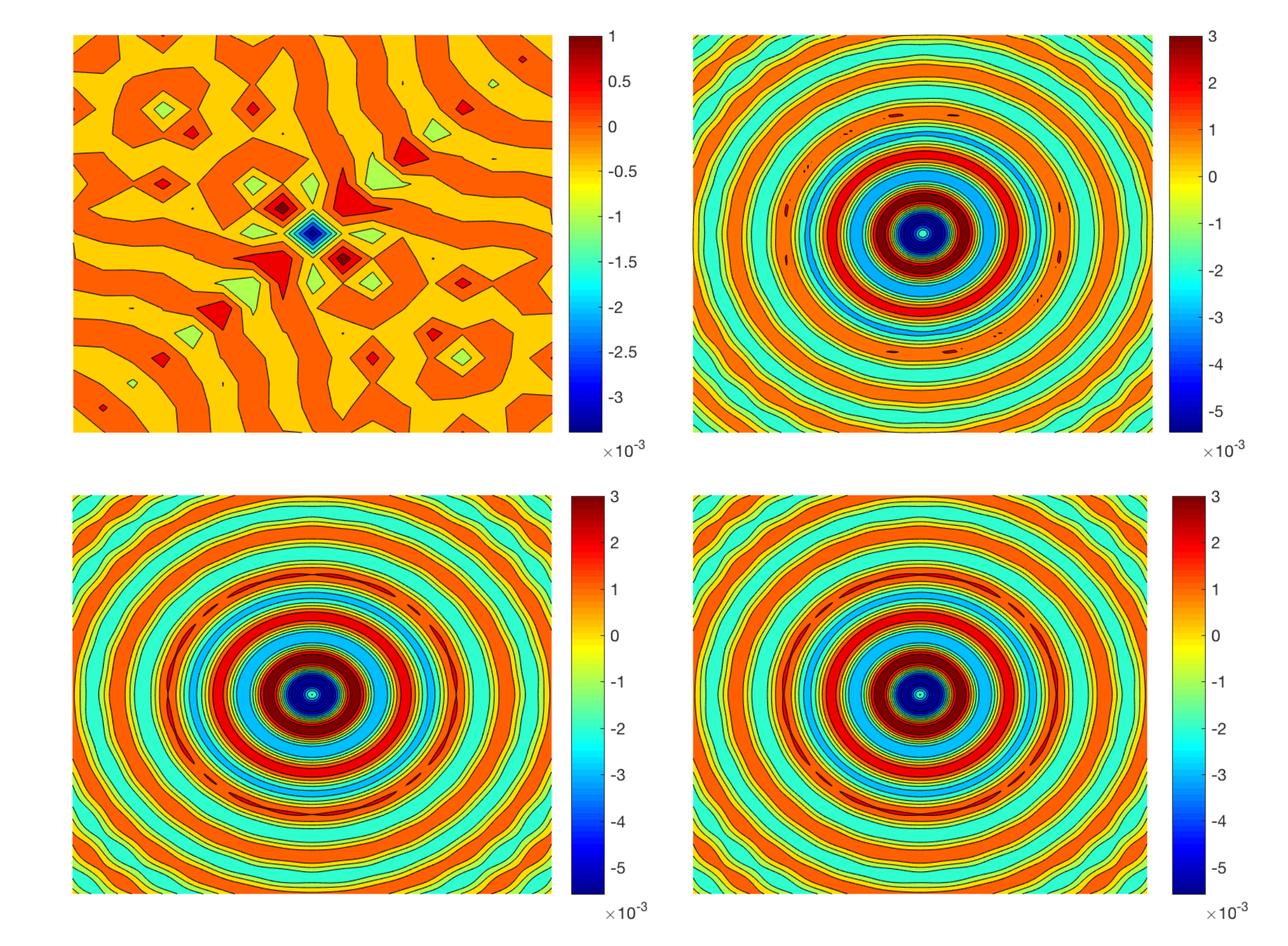}
\caption{The real part of the finite element solution for the equations~\eqref{num-eq2a}--\eqref{num-eq2b} for $k=60$ over $\T_m$ with $m=16$ (top left), $m=128$ (top right), $m=512$ (bottom left) and $m=1024$ (bottom right).}
\label{num-fig6}
\end{center}
\end{figure}

\begin{figure}[htbp]
\begin{center}
\includegraphics[width=0.9\textwidth]{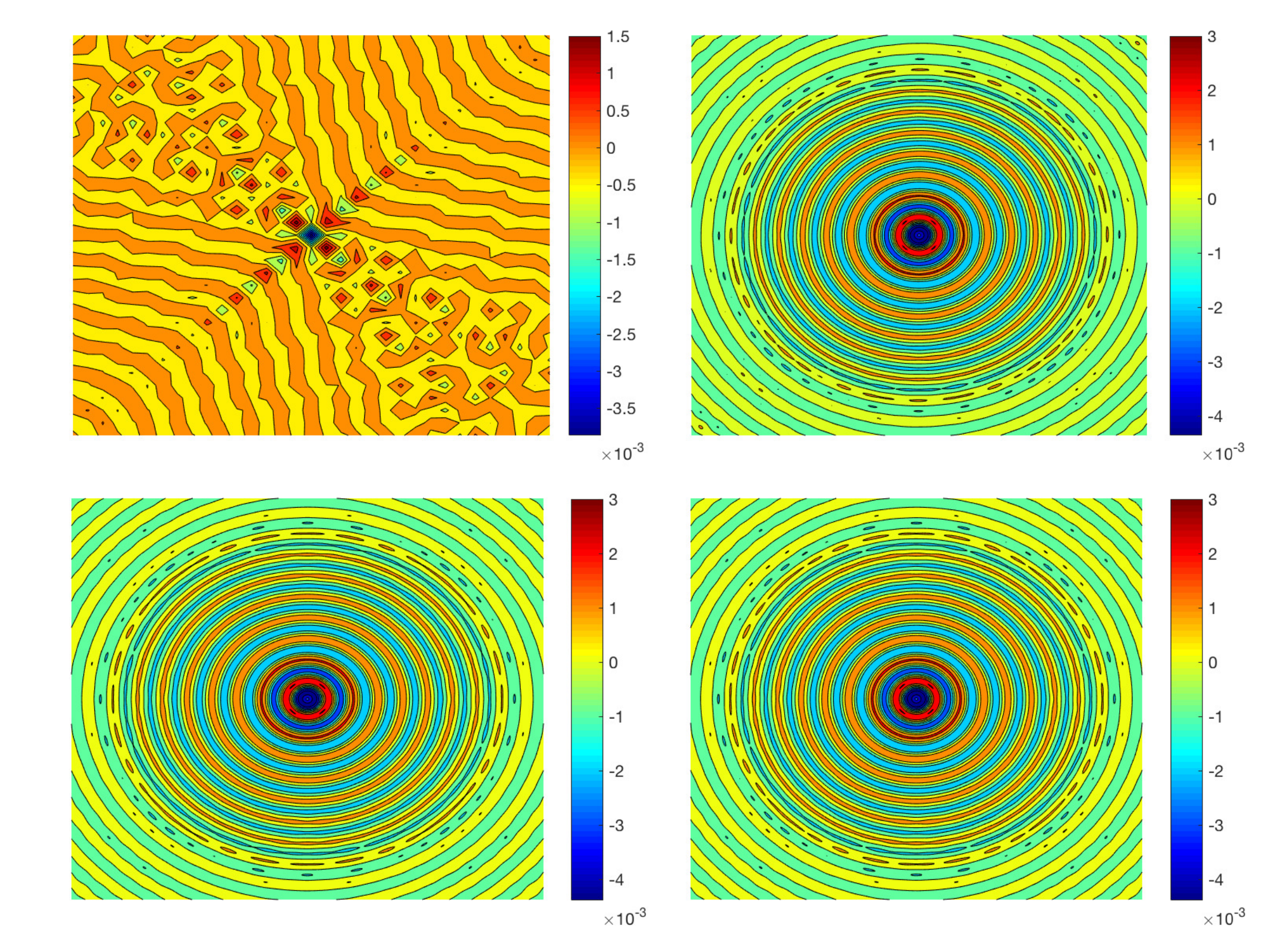}
\caption{The real part of the finite element solution for the equations~\eqref{num-eq2a}--\eqref{num-eq2b} for $k=120$ over $\T_m$ with $m=32$ (top left), $m=512$ (top right), $m=1024$ (bottom left) and $m=2048$ (bottom right).}
\label{num-fig7}
\end{center}
\end{figure}

\section{Concluding Remarks}
In this work, we have studied superconvergence properties of linear FEM based on PPR for the Helmholtz equation with large wave number.
We analyzed (1) gradient error between the finite element solution and the linear interpolation  $\|\nabla(u_h-u_I)\|_{L^2(\Om)}$ (c.f. \eqref{eq:th1eq1}) and
(2) the error between the true gradient and recovered gradient from the finite element solution  $\norm{\na u-G_hu_h}_{L^2(\Om)}$ (c.f. \eqref{ghpoll}) under the mesh condition $k(kh)^2 \le C_0$ (\ref{meshcond}). Both errors consist of two parts $C_1kh^{1+\alpha}+C_2k(kh)^2$ with
the first term improved by a factor $h^\alpha$ and the second term remained the same from the original gradient error.
We see that the recovered gradient still suffers from the pollution error.
We further analyzed (3) the difference between the finite element solution gradient
and the recovered gradient by PPR and found that the pollution part of this error can be improved to
$k(kh)^3$ (c.f. \eqref{eq:th2eq1}), which implies
{$\norm{G_hu_h-\na u_h}_0\ls kh$ if $k(kh)^2\le C_0$, (see remark~5.1). In another word,  $\norm{G_hu_h-\na u_h}_0$ can not provide a good measure of the $H^1$-error of the finite element solution for $h$ in the preasymptotic range since  $\norm{\na u-\na u_h}_0$ contains also the pollution term. However, the superconvergence rate $O(h^2)$ of the recovered gradient makes it possible that the Richardson extrapolation improves the numerical solution further. Therefore, $\norm{RG_hu_h-\na u_h}_0$ can measure the $H^1$-error of the finite element solution very well
and leads to asymptotically exact {\it a posteriori} error estimators. All aforementioned error bounds are
verified by numerical tests in Section~\ref{num}.
 As by-products, we also estimated the following quantities:
$\norm{G_hu_I-\na u}_{L^2(\Om)}$ (c.f. \eqref{eq:th0eq1}),  $\norm{\na P_hu-\na u_I}_{L^2(\Om)}$ (c.f. \eqref{eq:nabla-phu-ui}), $\norm{G_hP_hu-\na u}_{L^2(\Om)}$ (c.f., \eqref{phusup}),
and found that they have a common pollution term $(kh)^2$, which indicates that these quantities suffer much less from the pollution.}

\bibliographystyle{siam}
\bibliography{referrence}

\begin{thebibliography}{10}

\bibitem{comsol}
{\sc COMSOL AB.}, {\em COMSOL MultiPhysics User's Guide}, 3.5a~ed., 2008.

\bibitem{Ainsworth04}
{\sc M~Ainsworth}, {\em Discrete dispersion relation for hp-version finite
  element approximation at high wave number}, SIAM J. Numer. Anal., 42 (2004),
  pp.~553--575.

\bibitem{ak79}
{\sc A.K. Aziz and R.B. Kellogg}, {\em A scattering problem for the {H}elmholtz
  equation}, in Advances in Computer Methods for Partial Differential
  Equations-{III}, vol.~1, 1979, pp.~93--95.

\bibitem{bips95}
{\sc I.~Babu\v{s}ka, F.~Ihlenburg, E.T. Paik, and S.A. Sauter}, {\em A
  generalized finite element method for solving the {H}elmholtz equation in two
  dimensions with minimal pollution}, Comput. Methods Appl. Mech. Engrg., 128
  (1995), pp.~325--359.

\bibitem{bs00}
{\sc I.~Babu\v{s}ka and S.A. Sauter}, {\em Is the pollution effect of the {FEM}
  avoidable for the {H}elmholtz equation considering high wave numbers?}, SIAM
  Rev., 42 (2000), pp.~451--484.

\bibitem{bx03}
{\sc R.~E. Bank and J.~Xu}, {\em Asymptotically exact a posteriori error
  estimators, {P}art {I}: {G}rid with superconvergence}, SIAM J. Numer. Anal.,
  41 (2003), pp.~2294--2312.

\bibitem{br86}
{\sc H.~Blum and R.~Rannacher}, {\em Asymptotic error expansion and richardson
  extrapolation for linear finite elements}, Numer. Math., 49 (1986),
  pp.~11--38.

\bibitem{bs08}
{\sc S.C. Brenner and L.R. Scott}, {\em The mathematical theory of finite
  element methods}, Springer, New York, third~ed., 2008.

\bibitem{zbw}
{\sc E.~Burman, H.~Wu, and L.~Zhu}, {\em Continuous interior penalty finite
  element method for {H}elmholtz equation with high wave number: One
  dimensional analysis}, arXiv:1211.1424.

\bibitem{cx07}
{\sc L.~Chen and J.~Xu}, {\em Topics on adaptive finite element methods, in
  {A}daptive {C}omputations: {T}heory and {A}lgorithms, T. Tang and J. Xu,
  eds.}, Science Press, Beijing, 2007.

\bibitem{cx}
{\sc Z.~Chen and X.~Xiang}, {\em A source transfer domain decomposition method
  for helmholtz equations in unbounded domain}, SIAM J. Numer. Anal., 51
  (2013), pp.~2331--2356.

\bibitem{ciarlet78}
{\sc P.~G. Ciarlet}, {\em The finite element method for elliptic problems},
  North-Holland Pub. Co., New York, 1978.

\bibitem{dbb99}
{\sc A.~Deraemaeker, I.~Babu\v{s}ka, and P.~Bouillard}, {\em Dispersion and
  pollution of the {FEM} solution for the {H}elmholtz equation in one, two and
  three dimensions}, Internat. J. Numer. Methods Engrg., 46 (1999),
  pp.~471--499.

\bibitem{dss94}
{\sc J.~Douglas~Jr, J.E. Santos, and D.~Sheen}, {\em Approximation of scalar
  waves in the space-frequency domain}, Math. Models Methods Appl. Sci., 4
  (1994), pp.~509--531.

\bibitem{dw}
{\sc Y.~Du and H.~Wu}, {\em Preasymptotic error analysis of higher order {FEM}
  and {CIP-FEM} for {H}elmholtz equation with high wave number}, SIAM J. Numer.
  Anal., 53 (2015), pp.~782--804.

\bibitem{dzh}
{\sc Y.~Du and L.~Zhu}, {\em Preasymptotic error analysis of high order
  interior penalty discontinuous {G}alerkin methods for the {H}elmholtz
  equation with high wave number}, J. Sci. Comput., Accepted,  (2015).

\bibitem{em79}
{\sc B.~Engquist and A.~Majda}, {\em Radiation boundary conditions for acoustic
  and elastic wave calculations}, Comm. Pure Appl. Math., 32 (1979),
  pp.~313--357.

\bibitem{fw09}
{\sc X.~Feng and H.~Wu}, {\em Discontinuous {G}alerkin methods for the
  {H}elmholtz equation with large wave numbers}, SIAM J. Numer. Anal., 47
  (2009), pp.~2872--2896.

\bibitem{fw11}
\leavevmode\vrule height 2pt depth -1.6pt width 23pt, {\em $hp$-discontinuous
  {G}alerkin methods for the {H}elmholtz equation with large wave number},
  Math. Comp., 80 (2011), pp.~1997--2024.

\bibitem{harari97}
{\sc I.~Harari}, {\em Reducing spurious dispersion, anisotropy and reflection
  in finite element analysis of time-harmonic acoustics}, Comput. Meth. Appl.
  Mech. Engrg., 140 (1997), pp.~39--58.

\bibitem{h83}
{\sc P.~Helfrich}, {\em Asymptotic expansion for the finite element
  approximations of parabolic problems}, Bonner Math. Schriften, 158 (1983),
  pp.~11--30.

\bibitem{ib95a}
{\sc F.~Ihlenburg and I.~Babu\v{s}ka}, {\em Finite element solution of the
  {H}elmholtz equation with high wave number. {I}. {T}he {$h$}-version of the
  {FEM}}, Comput. Math. Appl., 30 (1995), pp.~9--37.

\bibitem{ib97}
\leavevmode\vrule height 2pt depth -1.6pt width 23pt, {\em Finite element
  solution of the {H}elmholtz equation with high wave number. {II}. {T}he
  {$h$}-{$p$} version of the {FEM}}, SIAM J. Numer. Anal., 34 (1997),
  pp.~315--358.

\bibitem{lmw}
{\sc A.~M. Lakhany, I.~Marek, and J.~R. Whiteman}, {\em Superconvergence
  results on mildly structured triangulations}, Comput. Methods Appl. Mech.
  Engrg., 189 (2000), pp.~1--75.

\bibitem{lsy98}
{\sc Q.~Lin, S.~Zhang, and N.~Yan}, {\em Asymptotic error expansion and defect
  correction for {S}obolev and viscoelasticity type equations}, J. Comput.
  Math., 16 (1998), pp.~57--62.

\bibitem{ms83}
{\sc G.~Marchuk and V.~Shaidurov}, {\em Difference Methods and Their
  Extrapolation}, Springer-Verlag, New York, 1983.

\bibitem{mps13}
{\sc JM~Melenk, A~Parsania, and S~Sauter}, {\em General {DG}-methods for highly
  indefinite {H}elmholtz problems}, Journal of Scientific Computing, 57 (2013),
  pp.~536--581.

\bibitem{ms10}
{\sc J.~M. Melenk and S.A. Sauter}, {\em Convergence analysis for finite
  element discretizations of the {H}elmholtz equation with
  {D}irichlet-to-{N}eumann boundary conditions}, Math. Comp., 79 (2010),
  pp.~1871--1914.

\bibitem{ms11}
\leavevmode\vrule height 2pt depth -1.6pt width 23pt, {\em Wavenumber explicit
  convergence analysis for {G}alerkin discretizations of the {H}elmholtz
  equation}, SIAM J. Numer. Anal., 49 (2011), pp.~1210--1243.

\bibitem{nz04}
{\sc A.~Naga and Z.~Zhang}, {\em A posteriori error estimates based on the
  polynomial preserving recovery}, SIAM J. Numer. Anal., 42 (2004),
  pp.~1780--1800.

\bibitem{sch74}
{\sc A.H. Schatz}, {\em An observation concerning {R}itz--{G}alerkin methods
  with indefinite bilinear forms}, Math.\ Comp., 28 (1974), pp.~959--962.

\bibitem{w89}
{\sc J.~Wang}, {\em Asymptotic expansions and $l^\infty$-error estimates for
  mixed finite element methods for second order elliptic problems}, Numer.
  Math., 55 (1989), pp.~401--430.

\bibitem{w}
{\sc H.~Wu}, {\em Pre-asymptotic error analysis of {CIP-FEM} and {FEM} for
  {H}elmholtz equation with high wave number. {P}art {I}: Linear version}, IMA
  J. Numer. Anal., 34 (2014), pp.~1266--1288.

\bibitem{wz07}
{\sc H.~Wu and Z.~Zhang}, {\em Can we have superconvergent gradient recovery
  under adaptive meshes?}, SIAM J. Numer. Anal., 45 (2007), pp.~1701--1722.

\bibitem{xz03}
{\sc J.~Xu and Z.~Zhang}, {\em Analysis of recovery type a posteriori error
  estimators for mildly structured grids}, Math. Comp., 73 (2003),
  pp.~1139--1152.

\bibitem{yz01}
{\sc N.~Yan and A.~Zhou}, {\em Gradient recovery type a posteriori error
  estimates for finite element approximations on irregular meshes}, Comput.
  Methods Appl. Mech. Engrg., 190 (2001), pp.~4289--4299.

\bibitem{z04t}
{\sc Z.~Zhang}, {\em Polynomial preserving gradient recovery and a posteriori
  estimate for bilinear element on irregular quadrilaterals}, Internat. J.
  Numer. Anal. Model., 1 (2004), pp.~1--24.

\bibitem{z04}
\leavevmode\vrule height 2pt depth -1.6pt width 23pt, {\em Polynomial
  preserving recovery for anisotropic and irregular grids}, J. Comput. Math.,
  22 (2004), pp.~331--340.

\bibitem{zl99}
{\sc Z.~Zhang and B.~Li}, {\em Analysis of a class of superconvergence patch
  recovery techniques for linear and bilinear finite elements}, Numer. Methods
  Partial Differential Equations, 15 (1999), pp.~151--167.

\bibitem{zn05}
{\sc Z.~Zhang and A.~Naga}, {\em A new finite element gradient recovery method:
  {S}uperconvergence property}, SIAM J. Sci. Comput., 26 (2005),
  pp.~1192--1213.

\bibitem{zd}
{\sc L.~Zhu and Y.~Du}, {\em Pre-asymptotic error analysis of $hp$-interior
  penalty discontinuous {G}alerkin methods for the {H}elmholtz equation with
  large wave number}, Comput. Math. Appl., 70 (2015), pp.~917--933.

\bibitem{zw}
{\sc L.~Zhu and H.~Wu}, {\em Pre-asymptotic error analysis of {CIP-FEM} and
  {FEM} for {H}elmholtz equation with high wave number. {P}art {II}: $hp$
  version}, SIAM J. Numer. Anal., 51 (2013), pp.~1828--1852.

\end{thebibliography}
\end{document}